\documentclass[final,leqno]{siamltex}

\usepackage{amsfonts,amssymb} 
\usepackage{amsmath} 
\usepackage{color} 
\usepackage{graphicx} 
\usepackage{epsfig,mathrsfs} 
\usepackage[bf,SL,BF]{subfigure} 
\usepackage{fancyhdr} 
\usepackage{CJK} 
\usepackage{caption} 
\usepackage{wrapfig} 
\usepackage{cases} 
\usepackage{bm}
\usepackage{enumerate}
\usepackage{url}
\usepackage{algorithm}
\usepackage{algorithmic}
\newtheorem{remark}{Remark}[section] 
\newtheorem{example}{Example}[section] 
\newtheorem{assumption}{Assumption}[section] 
\allowdisplaybreaks[4]
\title{Optimal convergence for the regularized solution of the model describing the competition between super- and sub- diffusions driven by fractional Brownian sheet noise
\thanks{This work was supported by the National Natural Science Foundation of China under Grant Nos. 12071195 and 12225107, and the Fundamental Research Funds for the Central Universities under Grant Nos. lzujbky-2021-it26 and lzujbky-2021-kb15.
}}

\author{Jing Sun\footnotemark[2] \and Daxin Nie\footnotemark[2]
\and Weihua Deng\footnotemark[2]\thanks{ School of Mathematics and Statistics, Gansu Key Laboratory of Applied Mathematics and Complex Systems, Lanzhou University, Lanzhou 730000, P.R. China (Email: dengwh@lzu.edu.cn).}
}
\begin{document}

\maketitle

\begin{abstract}
Super- and sub- diffusions are two typical types of anomalous diffusions in the natural world. In this work, we discuss the numerical scheme for the model describing the competition between super- and sub- diffusions driven by fractional Brownian sheet noise. Based on the obtained regulization result of the solution by using the properties of Mittag--Leffler function and the regularized noise by Wong--Zakai approximation, we make full use of the regularity of the solution operators to achieve optimal convergence of the regularized solution.  The spectral Galerkin method and the Mittag--Leffler Euler integrator are respectively used to deal with the space and time operators. In particular, by contour integral, the fast evaluation of the Mittag--Leffler Euler integrator is realized. We provide complete error analyses, which are verified by the numerical experiments.
\end{abstract}


\pagestyle{myheadings}
\thispagestyle{plain}
\markboth{Jing Sun et. al.}{CONVERGENCE FOR STOCHASTIC FRACTIONAL DIFFUSION EQUATION}

\section{Introduction}
Anomalous diffusions are ubiquitous in the natural world, two typical types of which are super- and sub- diffusions. The model to describe super-diffusion can be a time-change of Brownian motion by the stable subordinator, while the one to characterize sub-diffusion can be a time-change of Brownian motion by the inverse of the stable subordinator \cite{Deng.2020WS}. To  model the competition between super- and sub- diffusions, one can first time-change Brownian motion by the $s$-stable subordinator then further do the time-change by the inverse of the    $\alpha$-stable subordinator (the model (\ref{eqretosol}) concerned in this paper is to make two time-changes to killed Brownian motion).  Besides, the model (\ref{eqretosol}) also includes the deterministic source term $f(u)$ and fluctuation term (fractional Brownian sheet as the external noise).

Fractional Brownian sheet (fBs) can describe the anisotropic multi-dimensional data with self-similarity and long-range dependence \cite{Kamont.1996OtfaWf}. It has also been widely used in texture classification and often acts as driven force to stochastic differential equation. In this paper, we aim to propose an efficient fully discrete scheme for the following stochastic time-space fractional diffusion equation
\begin{equation}\label{eqretosol}
	\left\{\begin{aligned}
		&\partial_{t}u(x,t)+{}_{0}\partial_{t}^{1-\alpha}A^{s} u(x,t)=f(u)+\xi^{H_{1},H_{2}}(x,t)\qquad (x,t)\in{D}\times (0,T],\\
		&u(x,0)=0\qquad x\in{D},\\
		&u(x,t)=0\qquad (x,t)\in\partial{D}\times (0,T],
	\end{aligned}\right.
\end{equation}
where ${}_{0}\partial^{1-\alpha}_{t}$ with $\alpha\in(0,1)$ is the Riemann--Liouville fractional derivative, defined as \cite{Podlubny.1999FDE}
\begin{equation*}
	{}_{0}\partial^{1-\alpha}_{t}u=\frac{1}{\Gamma(\alpha)}\frac{\partial}{\partial t}\int_{0}^{t}(t-r)^{\alpha-1}u(r)dr;
\end{equation*}
the spectral fractional Laplacian $A^{s}$ with $s\in(0,1)$ is defined by
\begin{equation}\label{eqdefAs}
	A^{s}u=\sum_{k=1}^{\infty}\lambda_{k}^{s}(u,\phi_{k})\phi_{k}
\end{equation}
with $\{\lambda_{k},\phi_{k}\}_{k=1}^{\infty}$ being the non-decreasing eigenvalues and $L^{2}$-norm normalized eigenfunctions of operator $A=-\Delta$ satisfying a zero Dirichlet boundary condition; $f(u)$ is the nonlinear source term; the noise  $\xi^{H_{1},H_{2}}$ is defined by
\begin{equation}\label{eqdefxi}
	\xi^{H_{1},H_{2}}(x,t)=\frac{\partial^{2}W^{H_{1},H_{2}}(x,t)}{\partial x\partial t}
\end{equation}
with $W^{H_{1},H_{2}}(x,t)$ being a fBs on a stochastic basis $(\Omega,\mathcal{F},(\mathcal{F}_{t})_{t\in[0,T]},\mathbb{P})$ such that
\begin{equation*}
	\begin{aligned}
		&\mathbb{E}\left[W^{H_{1},H_{2}}(x,t)W^{H_{1},H_{2}}(y,r)\right]\\
         &=\frac{x^{2H_{1}}+y^{2H_{1}}-|x-y|^{2H_{1}}}{2}\\
		&~~ \times \frac{t^{2H_{2}}+r^{2H_{2}}-|t-r|^{2H_{2}}}{2},\quad(x,t),(y,r)\in {D}\times[0,T].
	\end{aligned}
\end{equation*}
Here $D=(0,1)$ is a bounded domain, the Hurst parameters $H_{1},H_{2}\in(0,\frac{1}{2}]$, and $\mathbb{E}$ means the expectation.

In recent years, a variety of numerical methods have been proposed for stochastic partial differential equations driven by Brownian sheet \cite{Anton.2020Afdaotoshe,Gyongy.1998LafsqppdedbswnI,Gyongy.1999LafsqppdedbswnI,Quer-Sardanyons.2006Ssfaswe}, but for fractional Brownian sheet, the numerical researches are relatively few. Among them,  in \cite{Cao.2017Aseewawarn}, the authors use the piecewise linear function to approximate the fractional Brownian sheet, and propose spatial semi-discrete scheme for classical diffusion equation and wave equation driven by fractional Brownian sheet, respectively; the paper \cite{Nie.2022Nafsnfdedbrn} adopts finite difference and finite element methods to discretize the time fractional diffusion equation driven by fractional Brownian sheet, but limited by the piecewise constant approximation for the fractional Brownian sheet noise, the spatial convergence is not optimal for small $\alpha$, i.e., the spatial convergence rate doesn't exactly coincides with the Sobolev regularity of the solution.

In this paper, we first study the regularity of the solution by means of Mittag--Leffler function, and then use the Wong--Zakai approximation to regularize $\xi^{H_{1},H_{2}}$, where the spectral method is adopted to approximate the  spatial direction and piecewise constant functions to temporal direction. Different from the previous regularization method used in \cite{Cao.2017Aseewawarn,Nie.2022Nafsnfdedbrn}, the regularization method provided in this paper can make full use of the properties of the solution operator and  achieve an optimal convergence of the regularized solution.  Next, the spectral Galerkin method and Mittag--Leffler Euler integrator are used to construct the numerical scheme of \eqref{eqretosol}. However, since the solution operators of \eqref{eqretosol} lack the semigroup property and the Mittag--Leffler function is composed of an infinite series, it will result in the significant storage costs and computational complexity in numerical simulation. So in what follows, with the help of contour integral (see \cite{Stenger.1981NmboWcosf} for the detailed introduction), an $\mathcal{O}(LM)$ fast Mittag--Leffler Euler integrator for temporal discretization is presented in Section 4,  where $M$ is the number of time steps and $2L+1$ is the number of integration points.

The rest of this paper is organized as follows. In Section 2, we first  recall some properties of
Mitttag--Leffler function and stochastic integral with respect to $\xi^{H_{1},H_{2}}$, and then propose the regularity of the solution. Next, we
provide the regularized equation of (1.1) and discuss the corresponding convergence in Section 3. In
Section 4, we construct the fully discrete scheme by using spectral Galerkin method and fast Mittag--Leffler integrator, and then the strict error estimates are presented. In Section 5, extensive numerical
experiments are performed to validate the effectiveness of our theoretical results. Finally, we conclude the paper with
some discussions in the last section. Throughout the paper, $C$ means a positive constant, whose value may vary from line
to line, $\|\cdot\|$ denotes the operator norm on $L^{2}(D)$,  $\epsilon>0$ is an  arbitrarily small quantity,
and $\mathbb{E}$ means the expectation.

\section{Regularity of the solution}\label{sec2}

In this section, we discuss the regularity of the solution of \eqref{eqretosol}. To begin with, we make the following assumption for $f(u)$.
\begin{assumption}\label{eqnonassump}
	Assume the nonlinear term $f(u)$ is a deterministic mapping satisfying
	\begin{equation}
		\begin{aligned}
			&\|f(u)\|_{L^{2}(D)}\leq C(1+\|u\|_{L^{2}(D)}),\\
			&\|f(u)-f(v)\|_{L^{2}(D)}\leq C\|u-v\|_{L^{2}(D)}
		\end{aligned}
	\end{equation}
	with $C$ being a positive constant.
\end{assumption}

 To obtain the expression of the mild solution, we need to introduce Mittag--Leffler function, i.e.,
\begin{equation*}
	E_{\alpha,\beta}(z)=\sum_{k=0}^{\infty}\frac{z^{k}}{\Gamma(k\alpha+\beta)}
\end{equation*}
for $\alpha>0$, $\beta\in\mathbb{R}$, $z\in \mathbb{C}$.

Then we provide the following lemmas for Mittag--Leffler function.
\begin{lemma}[see \cite{Kilbas.2006TaAoFDE}]\label{lemMit1}
	For $\alpha,\beta>0$, $\lambda\in\mathbb{C}$, $t^{\beta-1}E_{\alpha,\beta}(\lambda t^{\alpha})$ has the following properties:
	\begin{enumerate}[(1)]
		\item for $Re(z)>0$ and $|\lambda z^{-\alpha}|<1$, the Laplace transform of $t^{\beta-1}E_{\alpha,\beta}(\lambda t^{\alpha})$ can be represented as
		\begin{equation*}
			\int_{0}^{\infty}e^{-zt}t^{\beta-1}E_{\alpha,\beta}(\lambda t^{\alpha})dt=\frac{z^{\alpha-\beta}}{z^{\alpha}- \lambda};
		\end{equation*}
		\item for $\mu\in(0,1)$, there holds
		\begin{equation*}
			{}_{0}\partial_{t}^{\mu}(t^{\beta-1}E_{\alpha,\beta}(\lambda t^{\alpha}))=t^{\beta-\mu-1}E_{\alpha,\beta-\mu}(\lambda t^{\alpha});
		\end{equation*}
		\item for $\mu\in\mathbb{N}^{*}$, one has
		\begin{equation*}
			\partial_{t}^{\mu}(t^{\beta-1}E_{\alpha,\beta}(\lambda t^{\alpha}))=t^{\beta-\mu-1}E_{\alpha,\beta-\mu}(\lambda t^{\alpha}).
		\end{equation*}
	\end{enumerate}
\end{lemma}

\begin{lemma}\label{lemMit2}
	For $\alpha>0$, $\beta>-1$, $\sigma\in[0,s]$, $s\in(0,1)$, $t>0$, and $v\in L^{2}(D)$, we have
	\begin{equation*}
		\sum_{k=1}^{\infty}|\lambda_{k}^{\sigma}t^{\beta-1}E_{\alpha,\beta}(-\lambda_{k}^{s}t^{\alpha})(v,\phi_{k})|^{2}<Ct^{2\beta-2-2\alpha\sigma/s}\|v\|_{L^{2}(D)}^{2}.
	\end{equation*}
\end{lemma}
\begin{proof}
	By the facts $|E_{\alpha,\beta}(z)|\leq C(1+|z|)^{-1}$ for $\arg(z)\in[\nu,\pi]$ with $\nu\in(\frac{\alpha\pi}{2},\min(\pi,\alpha\pi))$ \cite{Kilbas.2006TaAoFDE}, one  obtains
	\begin{equation*}
		\begin{aligned}
			&\sum_{k=1}^{\infty}|\lambda_{k}^{\sigma}t^{\beta-1}E_{\alpha,\beta}(-\lambda_{k}^{s}t^{\alpha})(v,\phi_{k})|^{2}\\
			\leq&C\sum_{k=1}^{\infty}\left |\lambda_{k}^{\sigma}\frac{t^{\beta-1}}{1+\lambda^{s}_{k}t^{\alpha}}(v,\phi_{k})\right |^{2}\\
			\leq&Ct^{2\beta-2-2\alpha\sigma/s}\sup_{k\in\mathbb{N}^{*}}\left |\frac{(\lambda_{k}^{s}t^{\alpha})^{\sigma/s}}{1+\lambda^{s}_{k}t^{\alpha}}\right |^{2}\sum_{k=1}^{\infty}\left|(v,\phi_{k})\right |^{2}\\
			\leq&Ct^{2\beta-2-2\alpha\sigma/s}\|v\|_{L^{2}(D)}^{2},
		\end{aligned}
	\end{equation*}
	where we have used the fact that $\left |\frac{(\lambda_{k}^{s}t^{\alpha})^{\sigma/s}}{1+\lambda^{s}_{k}t^{\alpha}}\right |\leq C$ for $\sigma\in[0,s]$. Thus we complete the proof.
\end{proof}

Below we introduce the solution operators $\mathcal{S}(t)$ and $G(t,x,y)$ as
\begin{equation*}
	\begin{aligned}
		&\mathcal{S}(t)u=\sum_{k=1}^{\infty}E_{\alpha,1}(-\lambda^{s}_{k}t^{\alpha})(u,\phi_{k})\phi_{k},\\
		&G(t,x,y)=\sum_{k=1}^{\infty}G_{k}(t,x,y),
	\end{aligned}
\end{equation*}
where 
\begin{equation*}
	G_{k}(t,x,y)=E_{\alpha,1}(-\lambda^{s}_{k}t^{\alpha})\phi_{k}(x)\phi_{k}(y).
\end{equation*}
Thus, the solution of \eqref{eqretosol} can be represented as
\begin{equation}\label{equsolrep}
	u(x,t)=\int_{0}^{t}\mathcal{S}(t-r)f(u(x,r))dr+\int_{0}^{t}\int_{D}G(t-r,x,y)\xi^{H_{1},H_{2}}(y,r)dydr.
\end{equation}

Besides, we recall the following It\^{o} isometry for fractional Brownian sheet noise:

\begin{theorem}[see \cite{Nie.2022Nafsnfdedbrn}]\label{thmisometry}
	Assume $g_{1}(x,t)=g_{1,1}(x)g_{1,2}(t)$ and $g_{2}(x,t)=g_{2,1}(x)g_{2,2}(t)$ satisfying $g_{1,1}(x),g_{2,1}(x)\in H^{\frac{1-2H_{1}}{2}}_{0}(D)$ and $g_{1,2}(t),g_{2,2}(t)\in H^{\frac{1-2H_{2}}{2}}_{0}((0,T))$. Then one has
	\begin{equation*}
		\begin{aligned}
			&\mathbb{E}\left (\int_{0}^{T}\int_{{D}} g_{1}(x,t)\xi^{H_{1},H_{2}}(x,t)dxdt\int_{0}^{T}\int_{{D}} g_{2}(x,t)\xi^{H_{1},H_{2}}(x,t)dxdt\right )\\
			&\qquad\qquad\leq C\left \|{}_{0}\partial^{\frac{1-2H_{2}}{2}}_{t}g_{1,2}(t)\right \|_{L^{2}((0,T))}\left \|{}_{0}\partial^{\frac{1-2H_{2}}{2}}_{t}g_{2,2}(t)\right \|_{L^{2}((0,T))}\\
			&\qquad\qquad\qquad\times\|g_{1,1}(x)\|_{H^{\frac{1-2H_{1}}{2}}_{0}(D)}\|g_{2,1}(x)\|_{H^{\frac{1-2H_{1}}{2}}_{0}(D)}
		\end{aligned}
	\end{equation*}	
	and
	\begin{equation*}
		\begin{aligned}
			&\mathbb{E}\left (\int_{0}^{T}\int_{{D}} g_{1}(x,T-t)\xi^{H_{1},H_{2}}(x,t)dxdt\int_{0}^{T}\int_{{D}} g_{2}(x,T-t)\xi^{H_{1},H_{2}}(x,t)dxdt\right )\\
			&\qquad\qquad\leq C\left \|{}_{0}\partial^{\frac{1-2H_{2}}{2}}_{t}g_{1,2}(t)\right \|_{L^{2}((0,T))}\left \|{}_{0}\partial^{\frac{1-2H_{2}}{2}}_{t}g_{2,2}(t)\right \|_{L^{2}((0,T))}\\
			&\qquad\qquad\qquad\times\|g_{1,1}(x)\|_{H^{\frac{1-2H_{1}}{2}}_{0}(D)}\|g_{2,1}(x)\|_{H^{\frac{1-2H_{1}}{2}}_{0}(D)}.
		\end{aligned}
	\end{equation*}
	Here $H^{s}_{0}(D)=\left\{u\in H^{s}(\mathbb{R}),~u=0~in~\Omega^{c}\right \}$ with norm $\|u\|_{H^{s}_{0}(D)}=\|u\|_{L^{2}(D)}+|u|_{H^{s}_{0}(D)}$ and $|u|^{2}_{H^{s}_{0}(D)}=\int\int_{\mathbb{R}\times \mathbb{R}}\frac{(u(x)-u(y))^{2}}{|x-y|^{1+2s}}dxdy$; ${}_{0}\partial^{\alpha}_{t}$ is the Riemann--Liouville fractional derivative with $\alpha\in(0,1)$; when $\alpha=0$, ${}_{0}\partial^{0}_{t}$ denotes an identity operator.
\end{theorem}

\begin{remark}\label{Respace}
	For $q\in\mathbb{R}$, denote $\hat{H}^{2q}(D)=\mathbb{D}(A^{q})$ with norm $\|u\|_{\hat{H}^{2q}(D)}=\|A^{q}u\|_{L^{2}(D)}$ and $\mathbb{D}(A^{q})$ is the domain of $A^{q}$. It is well known that $H^{s}_{0}(D)$ and $\hat{H}^{s}(D)$ with $s\in (0,\frac{3}{2})$ are equivalent \cite{Bonito.2019NaotifL}. For $\mu\in(0,1/2)$, there holds $\frac{1}{C}\|{}_{0}\partial^{\mu}_{t} u\|_{L^{2}(0,T)}\leq |u|_{H^{\mu}_{0}(0,T)} \leq C\|{}_{0}\partial^{\mu}_{t} u\|_{L^{2}(0,T)}$  \cite{Ervin.2006Vfftsfade}.
\end{remark}

Now, we provide the spatial regularity of the solution.
\begin{theorem}\label{thmsoblevu}
	Let $u$ be the solution of \eqref{eqretosol} and $f(u)$ satisfy Assumption \ref{eqnonassump}. Letting $\frac{2sH_{2}}{\alpha}+H_{1}-1>0$ and $H_{1}+2s-1>0$, we have
	\begin{equation*}
		\mathbb{E}\|A^{\sigma}u\|^{2}_{L^{2}(D)}\leq C,
	\end{equation*}
	where $2\sigma\in[0,\min\{\frac{2sH_{2}}{\alpha}+H_{1}-1,H_{1}+2s-1\})$.
\end{theorem}
\begin{proof}
	According to \eqref{equsolrep}, one has
	\begin{equation*}
		\begin{aligned}
			\mathbb{E}\|A^{\sigma}u\|^{2}_{L^{2}(D)}\leq& C\mathbb{E}\left\|\int_{0}^{t}A^{\sigma}\mathcal{S}(t-r)f(u(r))dr\right \|_{L^{2}(D)}^{2}\\
			&+C\mathbb{E}\left\|\int_{0}^{t}\int_{D}A^{\sigma}G(t-r,x,y)\xi^{H_{1},H_{2}}(y,r)dydr\right \|_{L^{2}(D)}^{2}\\
			\leq& \uppercase\expandafter{\romannumeral1}+\uppercase\expandafter{\romannumeral2}.
		\end{aligned}
	\end{equation*}
	Using Lemma \ref{lemMit2}, Assumption \ref{eqnonassump}, and the Cauchy-Schwarz inequality, we obtain
	\begin{equation*}
		\begin{aligned}
			\uppercase\expandafter{\romannumeral1}
			\leq &C\mathbb{E}\left(\int_{0}^{t}	\left (\sum_{k=1}^{\infty}|\lambda_{k}^{\sigma}E_{\alpha,1}(-\lambda_{k}^{s}(t-r)^{\alpha})(f(u),\phi_{k})|^{2}\right )^{1/2}dr\right)^{2}\\
			\leq &C \mathbb{E}\left(\int_{0}^{t}(t-r)^{-\frac{\sigma\alpha}{s}}\|f(u(r))\|_{L^{2}(D)}dr\right)^{2}\\
			\leq &C \int_{0}^{t}(t-r)^{-\frac{2\sigma\alpha}{s}+1-\epsilon}\mathbb{E}\|f(u(r))\|_{L^{2}(D)}^{2}dr\\
			\leq&C\left (1+\int_{0}^{t}(t-r)^{-\frac{2\sigma\alpha}{s}+1-\epsilon}\mathbb{E}\|u\|_{L^{2}(D)}^{2}dr\right ),
		\end{aligned}
	\end{equation*}
	where we require $-\frac{2\sigma\alpha}{s}+1>-1$ 
to preserve the boundedness of $\uppercase\expandafter{\romannumeral1}$, i.e., $\sigma\leq s$.
	
	As for $\uppercase\expandafter{\romannumeral2}$, Theorem \ref{thmisometry} and the definition of $G$ lead to
	\begin{equation*}
		\uppercase\expandafter{\romannumeral2}\leq C\sum_{k=1}^{\infty}\int_{0}^{t}|\lambda_{k}^{\sigma}{}_{0}\partial^{\frac{1-2H_{2}}{2}}_{r}E_{\alpha,1}(-\lambda_{k}^{s}r^{\alpha})|^{2}\|\phi_{k}(y)\|_{H^{\frac{1-2H_{1}}{2}}_{0}(D)}^{2}\|\phi_{k}(x)\|_{L^{2}(D)}^{2}dr.
	\end{equation*}
	Thus combining the fact that $\lambda_{k}\geq Ck^{2}$ for $k\in \mathbb{N}^{*}$ \cite{Laptev.1997DaNepodiEs,Li.1983OtSeatep}, Lemmas \ref{lemMit1}, \ref{lemMit2}, Remark \ref{Respace}, and the properties of $\phi_{k}$, one obtains
	\begin{equation*}
		\begin{aligned}
			\uppercase\expandafter{\romannumeral2}\leq&C\int_{0}^{t}\sum_{k=1}^{\infty}|\lambda_{k}^{\sigma+\frac{1-2H_{1}}{4}}{}_{0}\partial^{\frac{1-2H_{2}}{2}}_{r}E_{\alpha,1}(-\lambda_{k}^{s}r^{\alpha})|^{2}\|\phi_{k}(x)\|_{L^{2}(D)}^{2}dr\\
			\leq&C\int_{0}^{t}\sum_{k=1}^{\infty}|\lambda_{k}^{\sigma+\frac{1-2H_{1}}{4}}r^{-\frac{1-2H_{2}}{2}}E_{\alpha,1-\frac{1-2H_{2}}{2}}(-\lambda_{k}^{s}r^{\alpha})|^{2}\|\phi_{k}(x)\|_{L^{2}(D)}^{2}dr\\
			\leq&C\int_{0}^{t}\sum_{k=1}^{\infty}|\lambda_{k}^{\sigma+\frac{1-H_{1}}{2}+\epsilon}r^{-\frac{1-2H_{2}}{2}}E_{\alpha,1-\frac{1-2H_{2}}{2}}(-\lambda_{k}^{s}r^{\alpha})(\lambda_{k}^{-\frac{1}{4}-\epsilon}\phi_{k}(x),\phi_{k}(x))|^{2}dr\\
			\leq&C\int_{0}^{t}r^{2-1+2H_{2}-2-2\alpha(\sigma+\frac{1-H_{1}}{2}+\epsilon)/s}dr.
		\end{aligned}
	\end{equation*}
	To preserve the boundedness of $\uppercase\expandafter{\romannumeral2}$, we need to require that $2H_{2}-2\alpha(\sigma+\frac{1-H_{1}}{2})/s>0$ and $\sigma+\frac{1-H_{1}}{2}<s$, i.e., $2\sigma<\min\{\frac{2sH_{2}}{\alpha}+H_{1}-1,H_{1}+2s-1\}$. Therefore, collecting the above estimates and using the Gr\"{o}nwall inequality \cite{Elliott.1992EewsandfafemftCe} result in the desired results.
\end{proof}

Next,  the following temporal regularity of the solution can be obtained.
\begin{theorem}\label{thmholderu}
	Let $f(u)$ satisfy Assumption  \ref{eqnonassump} and $H_{2}+\frac{\alpha(H_{1}-1)}{2s}>0$. Then the solution of \eqref{eqretosol} has the estimate
	\begin{equation*}
		\mathbb{E}\left \|\frac{u(t)-u(t-\tau)}{\tau^{\gamma}}\right \|^{2}_{L^{2}(D)}\leq C
	\end{equation*}
	with $\gamma\in\left(0,H_{2}+\frac{\alpha(H_{1}-1)}{2s}\right)$.
\end{theorem}
\begin{proof}
	Simple calculations show that
	\begin{equation*}
		\begin{aligned}
			&\mathbb{E}\left \|\frac{u(t)-u(t-\tau)}{\tau^{\gamma}}\right \|^{2}_{L^{2}(D)}\\
			\leq& C\mathbb{E}\left\|\frac{1}{\tau^{\gamma}}\Bigg(\int_{0}^{t}\mathcal{S}(t-r)f(u(r))dr-\int_{0}^{t-\tau}\mathcal{S}(t-\tau-r)f(u(r))dr\Bigg)\right \|_{L^{2}(D)}^{2}\\
			&+C\mathbb{E}\Bigg\|\frac{1}{\tau^{\gamma}}\Bigg(\int_{0}^{t}\int_{D}G(t-r,x,y)\xi^{H_{1},H_{2}}(y,r)dydr\\
			&\qquad-\int_{0}^{t-\tau}\int_{D}G(t-\tau-r,x,y)\xi^{H_{1},H_{2}}(y,r)dydr\Bigg)\Bigg \|_{L^{2}(D)}^{2}\\
			\leq& \uppercase\expandafter{\romannumeral1}+\uppercase\expandafter{\romannumeral2}.
		\end{aligned}
	\end{equation*}
	For $\uppercase\expandafter{\romannumeral1}$, it can be divided into two parts:
	\begin{equation*}
		\begin{aligned}
			\uppercase\expandafter{\romannumeral1}\leq& C\mathbb{E}\Bigg\|\frac{1}{\tau^{\gamma}}\int_{0}^{t-\tau}(\mathcal{S}(t-r)-\mathcal{S}(t-\tau-r))f(u)dr\Bigg\|^{2}_{L^{2}(D)}\\
			&+C\mathbb{E}\Bigg\|\frac{1}{\tau^{\gamma}}\int_{t-\tau}^{t}\mathcal{S}(t-r)f(u)dr\Bigg\|^{2}_{L^{2}(D)}\\
			\leq&\uppercase\expandafter{\romannumeral1}_{1}+\uppercase\expandafter{\romannumeral1}_{2}.
		\end{aligned}
	\end{equation*}
	Using the Cauchy-Schwarz inequality, Lemma \ref{lemMit2}, Assumption \ref{eqnonassump}, and Theorem \ref{thmsoblevu}, one  has
	\begin{align*}
			&\uppercase\expandafter{\romannumeral1}_{1}\\
			\leq&C\tau^{-2\gamma}\mathbb{E}\Bigg(\int_{0}^{t-\tau}\\
			&\quad\Bigg(\sum_{k=1}^{\infty}(E_{\alpha,1}(-\lambda_{k}^{s}(t-r)^{\alpha})-E_{\alpha,1}(-\lambda_{k}^{s}(t-\tau-r)^{\alpha}))^{2}(f(u),\phi_{k})^{2}\Bigg)^{1/2}dr\Bigg)^{2}\\
			\leq&C\tau^{-2\gamma}\mathbb{E}\int_{0}^{t-\tau}(t-\tau-r)^{1-\epsilon}\sum_{k=1}^{\infty}\left(\int_{t-\tau-r}^{t-r}-\lambda_{k}^{s}\eta^{\alpha-1}E_{\alpha,\alpha}(-\lambda_{k}^{s}\eta^{\alpha})d\eta\right )^{2}(f(u),\phi_{k})^{2}dr\\
			~\\
			\leq&C\tau^{1-2\epsilon-2\gamma}\mathbb{E}\int_{0}^{t-\tau}(t-\tau-r)^{1-\epsilon}\\
			&\qquad\qquad\qquad\quad\int_{t-\tau-r}^{t-r}\eta^{2\epsilon}\sum_{k=1}^{\infty}\left(-\lambda_{k}^{s}\eta^{\alpha-1}E_{\alpha,\alpha}(-\lambda_{k}^{s}\eta^{\alpha})\right )^{2}(f(u),\phi_{k})^{2}d\eta dr\\
			\leq& C\tau^{1-2\epsilon-2\gamma}\int_{0}^{t-\tau}(t-\tau-r)^{1-\epsilon}\int_{t-\tau-r}^{t-r}\eta^{-2+2\epsilon}\mathbb{E}\|f(u)\|_{L^{2}(D)}^{2}d\eta dr\\
			\leq&C\tau^{2-2\epsilon-2\gamma}\int_{0}^{t-\tau}(t-\tau-r)^{-1+\epsilon}dr,
	\end{align*}
	where $\gamma\in[0,1)$. Similarly, for $\gamma\in[0,1)$, we obtain
	\begin{equation*}
		\begin{aligned}
			\uppercase\expandafter{\romannumeral1}_{2}\leq& C\tau^{1-2\gamma}\mathbb{E}\int_{t-\tau}^{t}\sum_{k=1}^{\infty}(E_{\alpha,1}(-\lambda^{s}_{k}r^{\alpha})^{2}(f(u),\phi_{k}))^{2}dr\leq C.
		\end{aligned}
	\end{equation*}
	For $\uppercase\expandafter{\romannumeral2}$, similar to the derivation of $\uppercase\expandafter{\romannumeral1}$, there holds
	\begin{equation*}
		\begin{aligned}
			\uppercase\expandafter{\romannumeral2}\leq& C\mathbb{E}\Bigg\|\frac{1}{\tau^{\gamma}}\int_{0}^{t-\tau}\int_{D}(G(t-r,x,y)-G(t-\tau-r,x,y))\xi^{H_{1},H_{2}}(y,r)dydr\Bigg\|_{L^{2}(D)}^{2}\\
			&+C\mathbb{E}\Bigg\|\frac{1}{\tau^{\gamma}}\int_{t-\tau}^{t}\int_{D}G(t-r,x,y)\xi^{H_{1},H_{2}}(y,r)dydr\Bigg\|_{L^{2}(D)}^{2}\leq \uppercase\expandafter{\romannumeral2}_{1}+\uppercase\expandafter{\romannumeral2}_{2}.
		\end{aligned}
	\end{equation*}
	Theorem \ref{thmisometry} shows that
	\begin{equation*}
			\begin{aligned}
				\uppercase\expandafter{\romannumeral2}_{1}\leq&C\tau^{-2\gamma}\int_{0}^{t-\tau}\sum_{k=1}^{\infty}\left({}_{0}\partial^{\frac{1-2H_{2}}{2}}_{r}E_{\alpha,1}(-\lambda_{k}^{s}(r+\tau)^{\alpha})-{}_{0}\partial^{\frac{1-2H_{2}}{2}}_{r}E_{\alpha,1}(-\lambda_{k}^{s}r^{\alpha})\right)^{2}\\		
				&\quad \|\phi_{k}(y)\|_{H^{\frac{1-2H_{1}}{2}}_{0}(D)}^{2}\|\phi_{k}(x)\|_{L^{2}(D)}^{2}dr.\\
		\end{aligned}
	\end{equation*}
	Using Remark \ref{Respace} and the fractional Poincar\'e inequality \cite{Acosta.2017AfLeRosafea}, one has
	\begin{equation*}
		\begin{aligned}
			&\int_{0}^{t-\tau}\left({}_{0}\partial^{\frac{1-2H_{2}}{2}}_{r}E_{\alpha,1}(-\lambda_{k}^{s}(r+\tau)^{\alpha})-{}_{0}\partial^{\frac{1-2H_{2}}{2}}_{r}E_{\alpha,1}(-\lambda_{k}^{s}r^{\alpha})\right)^{2}dr\\
			\leq& C\|E_{\alpha,1}(-\lambda_{k}^{s}(r+\tau)^{\alpha})-E_{\alpha,1}(-\lambda_{k}^{s}(r)^{\alpha})\|_{H^{\frac{1-2H_{2}}{2}}_{0}(0,t-\tau)}^{2}\\
			\leq& C\|E_{\alpha,1}(-\lambda_{k}^{s}(r)^{\alpha})-E_{\alpha,1}(-\lambda_{k}^{s}(r-\tau)^{\alpha})\chi_{r>\tau}(r)\|_{H^{\frac{1-2H_{2}}{2}}_{0}(0,t)}^{2}\\
			\leq&C\int_{0}^{t}\left({}_{0}\partial^{\frac{1-2H_{2}}{2}}_{r}E_{\alpha,1}(-\lambda_{k}^{s}r^{\alpha})-{}_{0}\partial^{\frac{1-2H_{2}}{2}}_{r}E_{\alpha,1}(-\lambda_{k}^{s}(r-\tau)^{\alpha})\chi_{r>\tau}(r)\right)^{2}dr,\\
		\end{aligned}
	\end{equation*}
which leads to, for $\beta\in(0,1]$,
	\begin{equation*}
		\begin{aligned}
				\uppercase\expandafter{\romannumeral2}_{1}\leq&C\tau^{-2\gamma}\int_{0}^{t}\sum_{k=1}^{\infty}\|\phi_{k}(y)\|_{H^{\frac{1-2H_{1}}{2}}_{0}(D)}^{2}\|\phi_{k}(x)\|_{L^{2}(D)}^{2}\\
				&~\left(r^{-\frac{1-2H_{2}}{2}}E_{\alpha,1-\frac{1-2H_{2}}{2}}(-\lambda_{k}^{s}r^{\alpha})-(r-\tau)^{-\frac{1-2H_{2}}{2}}E_{\alpha,1-\frac{1-2H_{2}}{2}}(-\lambda_{k}^{s}(r-\tau)^{\alpha})\chi_{r-\tau>0}(r)\right)^{2} dr\\
				\leq&C\tau^{-2\gamma}\int_{0}^{\tau}\\
				&\quad\sum_{k=1}^{\infty}\left(-\lambda_{k}^{\frac{1-H_{1}}{2}+\epsilon}r^{-\frac{1-2H_{2}}{2}}E_{\alpha,1-\frac{1-2H_{2}}{2}}(-\lambda_{k}^{s}r^{\alpha})\right )^{2}(\lambda_{k}^{-\frac{1}{4}-\epsilon}\phi_{k},\phi_{k})^{2}d\eta dr\\
				&+C\tau^{\beta-2\gamma}\int_{\tau}^{t}\int_{r-\tau}^{r}\eta^{1-\beta}\\
				&\quad\sum_{k=1}^{\infty}\left(-\lambda_{k}^{\frac{1-H_{1}}{2}+\epsilon}\eta^{-1-\frac{1-2H_{2}}{2}}E_{\alpha,-\frac{1-2H_{2}}{2}}(-\lambda_{k}^{s}\eta^{\alpha})\right )^{2}(\lambda_{k}^{-\frac{1}{4}-\epsilon}\phi_{k},\phi_{k})^{2}d\eta dr\\
				\leq& \uppercase\expandafter{\romannumeral2}_{1,1}+\uppercase\expandafter{\romannumeral2}_{1,2}.
			\end{aligned}
		\end{equation*}
	Simple calculations result in 
	\begin{equation*}
		\begin{aligned}
				\uppercase\expandafter{\romannumeral2}_{1,1}\leq& C\tau^{-2\gamma}\int_{0}^{\tau}r^{-1+2H_{2}-2\alpha\left(\frac{1-H_{1}}{2}+\epsilon\right)/s}d\eta dr\\			
				\leq&C\tau^{2H_{2}-2\alpha(\frac{1-H_{1}}{2}+\epsilon)/s-2\gamma},
			\end{aligned}
		\end{equation*}
	where $2\gamma<2H_{2}-2\alpha\left(\frac{1-H_{1}}{2}\right)/s$. As for $\uppercase\expandafter{\romannumeral2}_{1,2}$, we have
		\begin{equation*}
			\begin{aligned}
				\uppercase\expandafter{\romannumeral2}_{1,2}\leq&  C\tau^{\beta-2\gamma}\int_{\tau}^{t}\int_{r-\tau}^{r}\eta^{1-\beta-1+2H_{2}-2-2\alpha\left(\frac{1-H_{1}}{2}+\epsilon\right)/s}d\eta dr\\
				\leq&C\tau^{\beta-2\gamma}\int_{0}^{t-\tau}r^{-1+\epsilon}\int_{r}^{r+\tau}\eta^{-\beta-1+2H_{2}-2\alpha\left(\frac{1-H_{1}}{2}+\epsilon\right)/s-\epsilon}d\eta dr,
			\end{aligned}
		\end{equation*}
		where $-\beta+2H_{2}-2\alpha\left(\frac{1-H_{1}}{2}+\epsilon\right)/s>0$ and $2\gamma\leq\beta+\min\left(1,-\beta+2H_{2}-2\alpha\left(\frac{1-H_{1}}{2}+\epsilon\right)/s \right)$ need to be satisfied, i.e., $\gamma<\min\{H_{2}+\frac{\alpha(H_{1}-1)}{2s},1+\epsilon\}$.
	Similarly for $\uppercase\expandafter{\romannumeral2}_{2}$, when $\gamma<H_{2}+\frac{\alpha(H_{1}-1)}{2s}$, it holds
	\begin{equation*}
		\begin{aligned}
			\uppercase\expandafter{\romannumeral2}_{2}\leq & C\frac{1}{\tau^{2\gamma}}\int_{0}^{\tau}\sum_{k=1}^{\infty}({}_{0}\partial^{\frac{1-2H_{2}}{2}}_{r}E_{\alpha,1}(-\lambda_{k}^{s}r^{\alpha}))^{2}\|\phi_{k}(y)\|_{H^{\frac{1-2H_{1}}{2}}_{0}(D)}^{2}\|\phi_{k}(x)\|_{L^{2}(D)}^{2}dr\\	\leq&C\frac{1}{\tau^{2\gamma}}\int_{0}^{\tau}\sum_{k=1}^{\infty}(\lambda_{k}^{\frac{1-H_{1}}{2}+\epsilon}r^{-\frac{1-2H_{2}}{2}}E_{\alpha,1-\frac{1-2H_{2}}{2}}(-\lambda_{k}^{s}r^{\alpha}))^{2}(\lambda_{k}^{-\frac{1}{4}-\epsilon}\phi_{k}(x),\phi_{k}(x))^{2}dr\\
			\leq&C\tau^{2H_{2}+\frac{\alpha(H_{1}-1-\epsilon)}{s}-2\gamma}\leq C.
		\end{aligned}
	\end{equation*}
	Combining the above estimates yields the desired results.
\end{proof}

\section{Wong--Zakai approximation for fractional Brownian sheet}
In recent years, several types of Wong--Zakai approximation have been proposed to regularize the noise \cite{Ikeda.1981SDEaDP,Sussmann.1978Otgbdasode,Wong.1965Otcooitsi,Wong.1965Otrboasde}. In this part, we use the spectral bases $\{\phi_{k}\}_{k=1}^{N}$ and piecewise constant function to approximate $\xi^{H_{1},H_{2}}(x,t)$, i.e.,
\begin{equation}\label{eqappronoise}
	\xi^{H_{1},H_{2}}_{R}(x,t)=\sum_{i=1}^{M}\sum_{k=1}^{N}\frac{1}{\tau}\int_{I_{i}}\int_{D}\phi_{k}(y)\xi^{H_{1},H_{2}}(y,r)dydr\phi_{k}(x)\chi_{I_{i}}(t),
\end{equation}
where $\chi_{I_{i}}$ is the characteristic function on $I_{i}=(t_{i-1},t_{i}]$ $(i=1,2,\ldots M)$, $t_{j}=j\tau$ $(j=0,1,2,\ldots,M)$, and $T=t_{M}$. Then we can introduce the regularized solution $u_{R}(x,t)$ satisfying
\begin{equation}\label{eqreglar}
	\left\{\begin{aligned}
		&\partial_{t}u_{R}(x,t)+{}_{0}\partial_{t}^{1-\alpha}A^{s} u_{R}(x,t)=f(u_{R})+\xi^{H_{1},H_{2}}_{R}(x,t)\qquad (x,t)\in{D}\times (0,T],\\
		&u_{R}(x,0)=0\qquad x\in{D},\\
		&u_{R}(x,t)=0\qquad (x,t)\in\partial{D}\times (0,T],
	\end{aligned}\right.
\end{equation}
To obtain the solution of the regularized equation \eqref{eqreglar}, we introduce
\begin{equation*}
	G_{R,N}(t,r,x,y)=\sum_{k=1}^{N}G_{R,k}(t,r,x,y),
\end{equation*}
where
\begin{equation*}
\begin{aligned}
		&G_{R,k}(t,r,x,y)=\mathcal{S}_{R,k}(t,r)\phi_{k}(x)\phi_{k}(y),\\
		&\mathcal{S}_{R,k}(t,r)=\frac{1}{\tau}\sum_{i=1}^{M}\chi_{I_{i}}(r)\int_{I_{i}}E_{\alpha,1}(-\lambda_{k}^{s}(t-\bar{r})^{\alpha})\chi_{(0,t)}(t-\bar{r})d\bar{r}.
\end{aligned}
\end{equation*}
Therefore, $u_{R}$ can be represented as
\begin{equation}\label{eqregularsol}
	u_{R}(x,t)=\int_{0}^{t}\mathcal{S}(t-r)f(u_{R})dr+\int_{0}^{T}\int_{D}G_{R,N}(t-r,x,y)\xi^{H_{1},H_{2}}(y,r)dydr.
\end{equation}

\begin{remark}
	From the derivation of \eqref{eqregularsol}, one can note that the approximation of the noise can be regarded as the approximation of solution operator $G(t,x,y)$. Different from the previous Wong--Zakai approximation provided in \cite{Cao.2017Aseewawarn,Nie.2022Nafsnfdedbrn}, our proposed approximation \eqref{eqappronoise} can  make full use of the regularity of solution operators, which leads to the optimal convergence order of the regularized solution (refer to Theorem \ref{theoremregularizaed1} for the details).
\end{remark}

\begin{lemma}\label{lemEkest}
Let  $k>0$, $\sigma\in[0,1]$, and $\gamma\geq 0$. If $\gamma+\frac{\alpha\sigma}{s}<\frac{1}{2}$, then there exists a uniform constant $C$ such that
	\begin{equation}
		\|{}_{0}\partial^{\gamma}_{t}\lambda^{\sigma}_{k}E_{\alpha,1}(-\lambda_{k}^{s}t^{\alpha})\|_{L^{2}((0,T))}\leq C.
	\end{equation}
\end{lemma}
\begin{proof}
	Lemma \ref{lemMit1} and the fact that $|E_{\alpha,\beta}(z)|\leq C(1+|z|)^{-1}$ for $\arg(z)\in[\nu,\pi]$ with $\nu\in(\frac{\alpha\pi}{2},\min(\pi,\alpha\pi))$ \cite{Kilbas.2006TaAoFDE} imply that
	\begin{equation*}
		\begin{aligned}
			&\|{}_{0}\partial^{\gamma}_{t}\lambda^{\sigma}_{k}E_{\alpha,1}(-\lambda_{k}^{s}t^{\alpha})\|_{L^{2}((0,T))}^{2}=\int_{0}^{T}\left({}_{0}\partial^{\gamma}_{t}\lambda^{\sigma}_{k}E_{\alpha,1}(-\lambda_{k}^{s}t^{\alpha})\right)^{2}dt\\
			&\quad =\int_{0}^{T}\left(\lambda^{\sigma}_{k}t^{-\gamma}E_{\alpha,1-\gamma}(-\lambda_{k}^{s}t^{\alpha})\right)^{2}dt
			\leq C\int_{0}^{T}\left (\frac{\lambda^{\sigma}_{k}t^{-\gamma}}{1+\lambda_{k}^{s}t^{\alpha}}\right )^{2}dt\\
			&\quad\leq C\int_{0}^{T}t^{-\frac{2\alpha\sigma}{s}-2\gamma}\left (\frac{(\lambda^{s}_{k}t^{\alpha})^{\frac{\sigma}{s}}}{1+\lambda_{k}^{s}t^{\alpha}}\right )^{2}dt\leq C\int_{0}^{T}t^{-\frac{2\alpha\sigma}{s}-2\gamma}dt,
		\end{aligned}
	\end{equation*}
	where $C$ is a positive constant independent of $k$. To preserve the boundedness of $\|{}_{0}\partial^{\gamma}_{t}\lambda^{\sigma}_{k}E_{k}(t)\|_{L^{2}((0,T))}^{2}$, we need to require $-2\gamma-\frac{2\alpha\sigma}{s}>-1$, i.e., $\gamma<\frac{1}{2}-\frac{\alpha\sigma}{s}$. Hence we complete the proof.
\end{proof}

Now we provide the convergence of Wong--Zakai approximation.
\begin{theorem}\label{theoremregularizaed1}
	Let $u$ and $u_{R}$ be the solutions of \eqref{eqretosol} and \eqref{eqreglar}, respectively. Let $f(u)$ satisfy Assumption \ref{eqnonassump}, $\frac{2sH_{2}}{\alpha}+H_{1}-1>0$, and $H_{1}+2s-1>0$. Then the following estimate  hold
	\begin{equation*}
		\mathbb{E}\|u(t)-u_{R}(t)\|^{2}_{L^{2}(D)}\leq C(N+1)^{-4\sigma}+CN^{1-2H_{2}}\tau^{2H_{2}-\frac{\alpha}{2s}-\epsilon},
	\end{equation*}
	where $2\sigma\in \left(0,\min\{\frac{2sH_{2}}{\alpha}+H_{1}-1,H_{1}+2s-1\}\right)$.
\end{theorem}
\begin{proof}
	From \eqref{equsolrep} and \eqref{eqregularsol}, we have
	\begin{equation*}
		\begin{aligned}
			&\mathbb{E}\|u(t)-u_{R}(t)\|^{2}_{L^{2}(D)}\\
			\leq&C\mathbb{E}\left \|\int_{0}^{t}\mathcal{S}(t-r)(f(u)-f(u_{R}))dr\right \|^{2}_{L^{2}(D)}\\
			&+C\mathbb{E}\left\|\int_{0}^{T}\int_{D}(G(t-r,x,y)\chi_{(0,t)}(t-r)-G_{R,N}(t,r,x,y))\xi^{H_{1},H_{2}}(y,r)dydr\right \|^{2}_{L^{2}(D)}\\
			\leq&\uppercase\expandafter{\romannumeral1}+\uppercase\expandafter{\romannumeral2}.
		\end{aligned}
	\end{equation*}
According to Assumption \ref{eqnonassump} and Lemma \ref{lemMit2}, there exists 
	\begin{equation*}
		\begin{aligned}
			\uppercase\expandafter{\romannumeral1}\leq &C\mathbb{E}\int_{0}^{t}\sum_{k=1}^{\infty}(E_{\alpha,1}(-\lambda_{k}^{s}(t-r)^{\alpha}))^{2}((f(u)-f(u_{R})),\phi_{k})^{2}ds\\
			\leq &C\int_{0}^{t}\mathbb{E}\|f(u)-f(u_{R})\|^{2}_{L^{2}(D)}dr\leq C\int_{0}^{t}\mathbb{E}\|u-u_{R}\|^{2}_{L^{2}(D)}dr.
		\end{aligned}
	\end{equation*}
	Then we can split $\uppercase\expandafter{\romannumeral2}$ into two parts:
	\begin{equation*}
		\begin{aligned}
			\uppercase\expandafter{\romannumeral2}\leq &C\mathbb{E}\sum_{k=1}^{N}\left\|\int_{0}^{T}\int_{D}(G_{k}(t-r,x,y)\chi_{(0,t)}(t-r)-G_{R,k}(t,r,x,y))\xi^{H_{1},H_{2}}(y,r)dydr\right \|^{2}_{L^{2}(D)}\\
			&+C\mathbb{E}\sum_{k=N+1}^{\infty}\left\|\int_{0}^{t}\int_{D}G_{k}(t-r,x,y)\xi^{H_{1},H_{2}}(y,r)dydr\right \|^{2}_{L^{2}(D)}\\
			\leq& \uppercase\expandafter{\romannumeral2}_{1}+\uppercase\expandafter{\romannumeral2}_{2}.
		\end{aligned}
	\end{equation*}
	Combining Lemma \ref{lemEkest}, Remark \ref{Respace}, the interpolation theorem \cite{Brenner.2008TMToFEM}, and $H_{2}\in(0,\frac{1}{2}]$, one has
	\begin{equation*}
		\begin{aligned}
			\uppercase\expandafter{\romannumeral2}_{1}\leq&C\sum_{k=1}^{N}\left\|{}_{0}\partial^{\frac{1-2H_{2}}{2}}_{r}(E_{\alpha,1}(-\lambda_{k}^{s}(t-r)^{\alpha})\chi_{(0,t)}(t-r)-\mathcal{S}_{R,k}(t,r)) \right \|^{2}_{L^{2}((0,T))}\left\|
			\phi_{k}(y)\right\|^{2}_{H^{\frac{1-2H_{1}}{2}}_{0}(D)}\\
			\leq&CN^{1-2H_{1}}\sum_{k=1}^{N}\lambda_{k}^{-\frac{1}{2}-2\epsilon}\left\|{}_{0}\partial^{\frac{1-2H_{2}}{2}}_{r}\lambda_{k}^{\frac{1}{4}+\epsilon}\left(E_{\alpha,1}(-\lambda_{k}^{s}(t-r)^{\alpha})\chi_{(0,t)}(t-r)-\mathcal{S}_{R,k}(t,r)\right)\right\|^{2}_{L^{2}((0,T))}\\
			\leq&CN^{1-2H_{1}}\tau^{2H_{2}-\frac{\alpha}{2s}-\epsilon}.
		\end{aligned}
	\end{equation*}
According to the property of eigenfunction $\phi_{k}$,  Remark \ref{Respace}, and Lemma \ref{lemMit2}, there exists
	\begin{equation*}
		\begin{aligned}
			\uppercase\expandafter{\romannumeral2}_{2}\leq& C\int_{0}^{t}\sum_{k=N+1}^{\infty}\left|{}_{0}\partial^{\frac{1-2H_{2}}{2}}_{r}E_{\alpha,1}(-\lambda_{k}^{s}r^{\alpha})\right|^{2}\|\phi_{k}(y)\|^{2}_{H^{\frac{1-2H_{1}}{2}}_{0}(D)}\|\phi_{k}(x)\|^{2}_{L^{2}(D)}dr\\
			\leq& C\lambda^{-2\sigma}_{N+1}\int_{0}^{t}\sum_{k=N+1}^{\infty}\left|\lambda_{k}^{\sigma+\frac{1-H_{1}}{2}+\epsilon}r^{-\frac{1-2H_{2}}{2}}E_{\alpha,1-\frac{1-2H_{2}}{2}}(-\lambda_{k}^{s}r^{\alpha})\right|^{2}\left(\lambda_{k}^{-\frac{1}{4}-\epsilon}\phi_{k}(x),\phi_{k}(x)\right)^{2}dr\\
			\leq& C\lambda^{-2\sigma}_{N+1}\int_{0}^{t}r^{2-1+2H_{2}-2-2(\sigma+\frac{1-H_{1}}{2}+\epsilon)\alpha/s}dr,
		\end{aligned}
	\end{equation*}
under the assumption $2H_{2}-2(\sigma+\frac{1-H_{1}}{2}+\epsilon)\alpha/s>0$ and $\sigma+\frac{1-H_{1}}{2}<s$, i.e., $2\sigma<\min\{\frac{2sH_{2}}{\alpha}+H_{1}-1,H_{1}+2s-1\}$.
Therefore, after gathering the above estimates, the desired result is reached.
\end{proof}

\section{The  Fully Discrete Scheme and Error Analyses}

In this section, we first use the spectral Galerkin method and Mittag--Leffler Euler integrator \cite{Dai.2022WaMEifsfSwfian,Kovacs.2020MEifasfoewan} to build the fully discrete scheme of \eqref{eqreglar}, and then speed up our algorithm with the help of contour integral provided in \cite{Stenger.1981NmboWcosf}.  At last, the strict error analyses are presented.
\subsection{The Fully Discrete Scheme}
Here, we introduce some notations first. Let $\mathbb{H}_{N}={\rm span}\{\phi_{1},\phi_{2},\ldots,\phi_{N}\}$ with $N\in\mathbb{N}^{*}$ be a finite dimensional subspace of $L^{2}(D)$ and $P_{N}:~L^{2}(D)\rightarrow \mathbb{H}_{N}$ the projection satisfying
\begin{equation*}
	(u,v_{N})=(P_{N}u,v_{N})\quad \forall v_{N}\in \mathbb{H}_{N}.
\end{equation*}
It is easy to verify that
\begin{equation*}
	P_{N}u=\sum_{j=1}^{N}(u,\phi_{j})\phi_{j}.
\end{equation*}
Introduce the discrete operator $A_{N}^{s}:\mathbb{H}_{N}\rightarrow\mathbb{H}_{N}$ as
\begin{equation*}
	(A^{s}_{N}u_{N},v_{N})=(A^{s}u_{N},v_{N})\quad \forall u_{N},v_{N}\in \mathbb{H}_{N}.
\end{equation*}
Then the spectral Galerkin scheme for \eqref{eqreglar} can be written as: find $u_{N}(t)\in \mathbb{H}_{N}$ satisfying
\begin{equation}\label{eqsemisch}
	\left\{
	\begin{aligned}
		&\partial_{t}u_{N}+{}_{0}\partial^{1-\alpha}_{t}A^{s}_{N}u_{N}=P_{N}f(u_{N})+\xi^{H_{1},H_{2}}_{R}(x,t)\qquad t\in(0,T],\\
		&u_{N}(0)=0.
	\end{aligned}
	\right.
\end{equation}
To show the representation of the solution of \eqref{eqsemisch}, we introduce
\begin{equation}\label{eqdefERN}
	\mathcal{S}_{R,N}(t)u=\sum_{k=1}^{N}E_{\alpha,1}(-\lambda_{k}^{s}t^{\alpha})(u,\phi_{k})\phi_{k}.
\end{equation}
Thus the solution of \eqref{eqsemisch} has the form
\begin{equation}\label{eqsemischsol}
	u_{N}(x,t)=\int_{0}^{t}\mathcal{S}_{R,N}(t-r)P_{N}f(u_{N})dr+\int_{0}^{T}\int_{D}G_{R,N}(t,r,x,y)\xi^{H_{1},H_{2}}(y,r)dydr.
\end{equation}
To obtain a fully discrete scheme, we approximate $f(u_{N})$ by piecewise constant function in \eqref{eqsemischsol}, which leads to the Mittag--Leffler Euler integrator, i.e.,
\begin{equation}\label{equfullscheme}
	\begin{aligned}
		\bar{u}^{n}_{N}
		=&\sum_{i=1}^{n}\int_{t_{i-1}}^{t_{i}}\mathcal{S}_{R,N}(t_{n}-r)P_{N}f(\bar{u}_{N}^{i-1})dr\\
		&+\sum_{i=1}^{n}\int_{t_{i-1}}^{t_{i}}\int_{D}G_{R,N}(t_{n},r,x,y)\xi^{H_{1},H_{2}}(y,r)dydr\\
		=&\sum_{i=1}^{n}\int_{t_{i-1}}^{t_{i}}\mathcal{S}_{R,N}(t_{n}-r)P_{N}f(\bar{u}_{N}^{i-1})dr\\
		&+\sum_{i=1}^{n}\int_{t_{i-1}}^{t_{i}}\int_{D}G_{R,S}(t_{n}-r,x,y)\xi_{R}^{H_{1},H_{2}}(y,r)dydr,
	\end{aligned}
\end{equation}
where $\bar{u}^{n}_{N}$ is the numerical solution at $t_{n}$, $\bar{u}^{0}_{N}=P_{N}u_{0}$, and
\begin{equation*}
	G_{R,S}(t,x,y)=\sum_{k=1}^{N}E_{\alpha,1}(-\lambda_{k}^{s}t^{\alpha})\phi_{k}(x)\phi_{k}(y).
\end{equation*}
\subsection{Fast Mittag--Leffler Euler integrator}
 
 Since the Mittag--Leffler function is composed of an infinite series,  a large amount of computations is needed to preserve the accuracy when simulating $\bar{u}^{n}_{N}$ $(n=1,2,\ldots,M)$ according to \eqref{equfullscheme}. On the other hand, the solution operators $\mathcal{S}_{R,N}$ and $G_{R,S}$ do not have semi-group property, which leads to an $\mathcal{O}(M^{2})$ computation complexity. So  an accurate and fast  Mittag--Leffler Euler integrator is provided by using the contour integral.

 Using the Laplace transform of Mittag--Leffler function and convolution property, one  obtains
\begin{equation}\label{mittagsolution}
	\begin{aligned}
		\bar{u}_{N}^{n}
		=&\sum_{i=1}^{n}\sum_{k=1}^{N}\frac{1}{2\pi\mathbf{i}}\int_{t_{i-1}}^{t_{i}}\int_{\Re(z)=\vartheta}e^{z(t_{n}-r)}z^{\alpha-1}(z^{\alpha}+\lambda_{k}^{s})^{-1}dz(P_{N}f(\bar{u}_{N}^{i-1}),\phi_{k})\phi_{k}dr\\
		&+\sum_{i=1}^{n}\sum_{k=1}^{N}\frac{1}{2\pi\mathbf{i}}\int_{t_{i-1}}^{t_{i}}\int_{\Re(z)=\vartheta}e^{z(t_{n}-r)}z^{\alpha-1}(z^{\alpha}+\lambda_{k}^{s})^{-1}dz(\xi^{H_{1},H_{2}}_{R}(x,r),\phi_{k})\phi_{k}dr,
	\end{aligned}
\end{equation}
where $\vartheta$ is a positive constant, $\mathbf{i}$ is the imaginary unit, and $\Re(z)$ means the real part of $z$.

  From \eqref{mittagsolution}, it's easy to find  that the integrand functions are analytic for $z\in\Sigma_{\theta}=\{z\in \mathbb{C}, |z|>0, \arg(z)<\theta\}$ with $\theta\in (\frac{\pi}{2},\pi)$, so one can deform the contour $\Re(z)=\vartheta$  to $\Gamma_{\mu}$, i.e.,
\begin{equation*}
	\Gamma_{\mu}=\{\mu(1-\sin(\nu+\mathbf{i}r)):r\in\mathbb{R}\}
\end{equation*}
with $\nu\in(0,\theta-\frac{\pi}{2})$ and $\mu$ being a parameter to be determined in the following. Thus $\bar{u}^{n}_{N}$ can be rewritten as
\begin{equation*}
	\begin{aligned}
		\bar{u}_{N}^{n}
		=&\sum_{i=1}^{n}\sum_{k=1}^{N}\frac{1}{2\pi\mathbf{i}}\int_{t_{i-1}}^{t_{i}}\int_{\Gamma_{\mu}}e^{z(t_{n}-r)}z^{\alpha-1}(z^{\alpha}+\lambda_{k}^{s})^{-1}dz(P_{N}f(\bar{u}_{N}^{i-1}),\phi_{k})\phi_{k}dr\\
		&+\sum_{i=1}^{n}\sum_{k=1}^{N}\frac{1}{2\pi\mathbf{i}}\int_{t_{i-1}}^{t_{i}}\int_{\Gamma_{\mu}}e^{z(t_{n}-r)}z^{\alpha-1}(z^{\alpha}+\lambda_{k}^{s})^{-1}dz(\xi^{H_{1},H_{2}}_{R}(x,r),\phi_{k})\phi_{k}dr.
	\end{aligned}
\end{equation*}
After exchanging the order of the integration, it yields that
\begin{equation}\label{eqfullsch00}
	\begin{aligned}
		\bar{u}_{N}^{n}
		=&\sum_{i=1}^{n}\sum_{k=1}^{N}\frac{1}{2\pi\mathbf{i}}\int_{\Gamma_{\mu}}(e^{zt_{n-i+1}}-e^{zt_{n-i}})z^{\alpha-2}(z^{\alpha}+\lambda_{k}^{s})^{-1}dz(P_{N}f(\bar{u}_{N}^{i-1}),\phi_{k})\phi_{k}\\
		&+\sum_{i=1}^{n}\sum_{k=1}^{N}\frac{1}{2\pi\mathbf{i}}\int_{\Gamma_{\mu}}(e^{zt_{n-i+1}}-e^{zt_{n-i}})z^{\alpha-2}(z^{\alpha}+\lambda_{k}^{s})^{-1}dz(\xi^{H_{1},H_{2}}_{R}(x,r),\phi_{k})\phi_{k}.
	\end{aligned}
\end{equation}
To simulate $\bar{u}_{N}^{n}$ effectively, we need to provide an efficient algorithm to calculate $\frac{1}{2\pi\mathbf{i}}\int_{\Gamma_{\mu}}e^{zt}z^{\alpha-2}(z^{\alpha}+\lambda_{k}^{s})^{-1}dz$. According to \cite{Stenger.1981NmboWcosf}, we have the approximation
\begin{equation}\label{eqdefErrork}
	\frac{1}{2\pi\mathbf{i}}\int_{\Gamma_{\mu}}e^{zt}z^{\alpha-2}(z^{\alpha}+\lambda_{k}^{s})^{-1}dz=\sum_{j=-L}^{L}\omega_{j}e^{z_{j}t}z_{j}^{\alpha-2}(z_{j}^{\alpha}+\lambda_{k}^{s})^{-1}+\mathcal{E}_{k}(t)
\end{equation}
with
\begin{equation*}
	\omega_{j}=-\frac{\bar{h}}{2\pi\mathbf{i}}\rho^{\prime}(j\bar{h}), \quad z_{h}=\rho(j\bar{h}),\quad \bar{h}=\sqrt{\frac{2\pi q}{ L}},
\end{equation*}
and
\begin{equation*}
	\rho(r)=\mu(1-\sin(\nu+\mathbf{i}r)).
\end{equation*}
Here $\rho'$ means the derivative of $\rho$. Substituting \eqref{eqdefErrork} into \eqref{eqfullsch00} leads to the fast fully discrete scheme
\begin{equation}\label{equfullschemefast}
	\begin{aligned}
		u_{N}^{n}=&\sum_{i=1}^{n}\sum_{k=1}^{N}\Bigg(\sum_{j=-L}^{L}\omega_{j}(e^{z_{j}t_{n-i+1}}-e^{z_{j}t_{n-i}})z_{j}^{\alpha-2}(z_{j}^{\alpha}+\lambda_{k}^{s})^{-1}\Bigg)(P_{N}f(u^{i-1}_{N}),\phi_{k})\phi_{k}\\
		&+\sum_{i=1}^{n}\sum_{k=1}^{N}\Bigg(\sum_{j=-L}^{L}\omega_{j}(e^{z_{j}t_{n-i+1}}-e^{z_{j}t_{n-i}})z_{j}^{\alpha-2}(z_{j}^{\alpha}+\lambda_{k}^{s})^{-1}\Bigg)(\xi^{H_{1},H_{2}}_{R}(x,t_{i}),\phi_{k})\phi_{k}\\
		=&\sum_{j=-L}^{L}\sum_{k=1}^{N}\Bigg(e^{z_{j}\tau}\mathcal{G}_{1,k,j}^{n-1}+\omega_{j}(e^{z_{j}\tau}-1)z_{j}^{\alpha-2}(z_{j}^{\alpha}+\lambda_{k}^{s})^{-1}(P_{N}f(u^{n-1}_{N}),\phi_{k})\phi_{k}\Bigg)\\
		&+\sum_{j=-L}^{L}\sum_{k=1}^{N}\Bigg(e^{z_{j}\tau}\mathcal{G}_{2,k,j}^{n-1}+\omega_{j}(e^{z_{j}\tau}-1)z_{j}^{\alpha-2}(z_{j}^{\alpha}+\lambda_{k}^{s})^{-1}(\xi^{H_{1},H_{2}}_{R}(x,t_{n}),\phi_{k})\phi_{k}\Bigg),
	\end{aligned}
\end{equation}
where the history terms $\mathcal{G}_{1,k,j}^{n}$ and $\mathcal{G}_{2,k,j}^{n}$ are defined by
\begin{equation}\label{eqdefhis}
	\begin{aligned}
		\mathcal{G}_{1,k,j}^{n}=&\sum_{i=0}^{n-1}\omega_{j}(e^{z_{j}t_{n-i+1}}-e^{z_{j}t_{n-i}})z_{j}^{\alpha-2}(z_{j}^{\alpha}+\lambda_{k}^{s})^{-1}(P_{N}f(u^{i-1}_{N}),\phi_{k})\phi_{k},\\
		\mathcal{G}_{2,k,j}^{n}=&\sum_{i=0}^{n-1}\omega_{j}(e^{z_{j}t_{n-i+1}}-e^{z_{j}t_{n-i}})z_{j}^{\alpha-2}(z_{j}^{\alpha}+\lambda_{k}^{s})^{-1}(\xi^{H_{1},H_{2}}_{R}(x,t_{i}),\phi_{k})\phi_{k}.	
	\end{aligned}
\end{equation}
\begin{remark}
	Compared with the algorithm provided in \cite{Kovacs.2020MEifasfoewan}, we can simulate \eqref{equfullscheme} in Laplace domain with fewer integration points, which avoids the huge computational cost on solving the  Mittag--Leffler function.  On the other hand, we don't need to calculate the sum of $\int_{t_{i-1}}^{t_{i}}\mathcal{S}_{RN}(t-r)P_{N}f(\bar{u}_{N}^{i-1})dr$ and $\int_{t_{i-1}}^{t_{i}}\int_{D}G_{RS}(t-r,x,y)\xi_{R}^{H_{1},H_{2}}(y,r)dydr$ in each iteration, which reduces the computation complexity from $\mathcal{O}(M^{2})$ to $\mathcal{O}(LM)$, where $M$ is the number of time step and $2L+1$ is the number of integration points.
\end{remark}

\subsection{Spatial error analyses}
To propose the convergence in spatial direction, we consider the spectral Galerkin semi-discrete scheme, i.e., find $\bar{u}_{N}\in \mathbb{H}_{N}$ satisfying
\begin{equation}\label{eqsemisch0}
	\left \{\begin{aligned}
		&\partial_{t}\bar{u}_{N}+{}_{0}\partial^{1-\alpha}_{t}A^{s}_{N}\bar{u}_{N}=P_{N}f(\bar{u}_{N})+\xi^{H_{1},H_{2}}_{RS}(x,t)\qquad t\in(0,T],\\
		&\bar{u}_{N}(0)=0,
	\end{aligned}\right .
\end{equation}
where
\begin{equation*}
	\xi^{H_{1},H_{2}}_{RS}(x,t)=\sum_{j=1}^{N}\int_{D}\phi_{j}(y)\xi^{H_{1},H_{2}}(y,r)dy\phi_{j}(x).
\end{equation*}
Thus the solution of \eqref{eqsemisch0} can be written as
\begin{equation}\label{eqsemisch0sol}
	\bar{u}_{N}=\int_{0}^{t}\mathcal{S}_{R,N}(t-r)P_{N}f(\bar{u}_{N})dr+\int_{0}^{t}\int_{D}G_{R,S}(t-r,x,y)\xi_{RS}^{H_{1},H_{2}}(y,r)dydr.
\end{equation}

Similar to the proofs of Theorems \ref{thmsoblevu} and \ref{thmholderu}, we have the following estimates.

\begin{theorem}\label{thmregnumsol}
	Let $\bar{u}_{N}$ be the solution of \eqref{eqsemisch0}. If $f(u)$ satisfies Assumption \ref{eqnonassump}, then there holds
	\begin{equation*}
		\begin{aligned}
			&\mathbb{E}\left \|\bar{u}_{N}(t)\right \|^{2}_{L^{2}(D)}\leq C,\\
			&\mathbb{E}\left \|\frac{\bar{u}_{N}(t)-\bar{u}_{N}(t-\tau)}{\tau^{\gamma}}\right \|^{2}_{L^{2}(D)}\leq C
		\end{aligned}
	\end{equation*}
	with $\gamma\in\left(0,H_{2}+\frac{\alpha(H_{1}-1)}{2s}\right)$.
\end{theorem}

\begin{theorem}\label{thmspacon}
	Let $u$ and $\bar{u}_{N}$ be the solutions of \eqref{eqretosol} and \eqref{eqsemisch0}, respectively. If Assumption \ref{eqnonassump} is satisfied, then we have
	\begin{equation*}
		\mathbb{E}\|u-\bar{u}_{N}\|^{2}_{L^{2}(D)}\leq C(N+1)^{-4\sigma},
	\end{equation*}
	where $2\sigma<\min\{\frac{2sH_{2}}{\alpha}+H_{1}-1,2H_{1}+2s-1\}$.
\end{theorem}
\begin{proof}
	According to \eqref{equsolrep} and \eqref{eqsemisch0sol}, we can split  $\mathbb{E}\|u-\bar{u}_{N}\|^{2}_{L^{2}(D)}$ into two parts
	\begin{equation*}
		\begin{aligned}
			&\mathbb{E}\|u-\bar{u}_{N}\|^{2}_{L^{2}(D)}\\
			\leq&C\mathbb{E}\left\|\int_{0}^{t}\mathcal{S}(t-r)f(u)-\mathcal{S}_{R,N}(t-r)P_{N}f(\bar{u}_{N})dr\right\|^{2}_{L^{2}(D)}\\
			&+C\mathbb{E}\Bigg\|\int_{0}^{t}\int_{D}G(t-r,x,y)\xi^{H_{1},H_{2}}(y,r)dydr-\int_{0}^{t}\int_{D}G_{R,S}(t-r,x,y)\xi_{RS}^{H_{1},H_{2}}(y,r)dydr\Bigg \|^{2}_{L^{2}(D)}\\
			\leq&\uppercase\expandafter{\romannumeral1}+\uppercase\expandafter{\romannumeral2}.
		\end{aligned}
	\end{equation*}
	Simple calculations lead to
	\begin{equation*}
		\begin{aligned}
			\uppercase\expandafter{\romannumeral1}\leq& C\mathbb{E}\left\|\int_{0}^{t}(\mathcal{S}(t-r)-\mathcal{S}_{R,N}(t-r)P_{N})f(u)dr\right \|^{2}_{L^{2}(D)}\\
			&+C\mathbb{E}\left\|\int_{0}^{t}\mathcal{S}_{R,N}(t-r)P_{N}(f(u)-f(\bar{u}_{N}))dr\right \|^{2}_{L^{2}(D)}\leq \uppercase\expandafter{\romannumeral1}_{1}+\uppercase\expandafter{\romannumeral1}_{2}.
		\end{aligned}
	\end{equation*}
	As for $\uppercase\expandafter{\romannumeral1}_{1}$, the definitions of $\mathcal{S}(t)$ and $\mathcal{S}_{R,N}(t){P_{N}}$ yield that for $v\in L^{2}(D)$ and $\sigma\in[0,s]$,
	\begin{equation*}
		\begin{aligned}
			&\|(\mathcal{S}(t)-\mathcal{S}_{R,N}(t)P_{N})v\|^{2}\\
			\leq&C\sum_{k=N+1}^{\infty}(E_{\alpha,1}(-\lambda_{k}^{s}r^{\alpha}))^{2}(v,\phi_{k})^{2}_{L^{2}(D)}\\
			\leq&C\lambda_{N+1}^{-2\sigma}\sum_{k=N+1}^{\infty}(\lambda_{k}^{\sigma}E_{\alpha,1}(-\lambda_{k}^{s}r^{\alpha}))^{2}(v,\phi_{k})^{2}\\
			\leq& C\lambda_{N+1}^{-2\sigma}t^{-2\sigma\alpha/s}\|v\|_{L^{2}(D)}^{2},
		\end{aligned}
	\end{equation*}
	which leads to
	\begin{equation*}
		\begin{aligned}
			\uppercase\expandafter{\romannumeral1}_{1}\leq& C\lambda_{N+1}^{-2\sigma}\mathbb{E}\Bigg(\int_{0}^{t}(t-r)^{-\sigma\alpha/s}\|f(u)\|_{L^{2}(D)}dr\Bigg)^{2}\\
			\leq&C\lambda_{N+1}^{-2\sigma}\int_{0}^{t}(t-r)^{1-2\sigma\alpha/s-\epsilon}\mathbb{E}\|f(u)\|_{L^{2}(D)}^{2}dr\\
			\leq&C\lambda_{N+1}^{-2\sigma}\int_{0}^{t}(t-r)^{1-2\sigma\alpha/s-\epsilon}\mathbb{E}\|u\|_{L^{2}(D)}^{2}dr.
		\end{aligned}
	\end{equation*}
	As for $\uppercase\expandafter{\romannumeral1}_{2}$, Assumption \ref{eqnonassump} and Lemma \ref{lemMit2} show that
	\begin{equation*}
		\begin{aligned}
			\uppercase\expandafter{\romannumeral1}_{2}\leq& C\mathbb{E}\int_{0}^{t}\sum_{k=1}^{\infty}E_{\alpha,1}(-\lambda_{k}^{s}(t-r)^{\alpha})^{2}(f(u)-f(\bar{u}_{N}),\phi_{k})^{2}dr\\
			\leq&C\mathbb{E}\int_{0}^{t}\sum_{k=1}^{\infty}\|f(u)-f(\bar{u}_{N})\|_{L^{2}(D)}^{2}dr\leq C\int_{0}^{t}\sum_{k=1}^{\infty}\mathbb{E}\|u-\bar{u}_{N}\|_{L^{2}(D)}^{2}dr.
		\end{aligned}
	\end{equation*}
	By the definitions of $G$ and $G_{R,S}$, it holds
	\begin{equation*}
		\begin{aligned}
			\uppercase\expandafter{\romannumeral2}\leq& C\int_{0}^{t_{n}}\sum_{k=N+1}^{\infty}\left|{}_{0}\partial^{\frac{1-2H_{2}}{2}}_{r}E_{\alpha,1}(-\lambda_{k}^{s}r^{\alpha})\right|^{2}\|\phi_{k}(y)\|^{2}_{H^{\frac{1-2H_{1}}{2}}_{0}(D)}\|\phi_{k}(x)\|^{2}_{L^{2}(D)}dr\\
			\leq& C\lambda^{-2\sigma}_{N+1}\int_{0}^{t_{n}}\sum_{k=N+1}^{\infty}\left|\lambda_{k}^{\sigma+\frac{1-H_{1}}{2}+\epsilon}r^{-\frac{1-2H_{2}}{2}}E_{\alpha,1-\frac{1-2H_{2}}{2}}(-\lambda_{k}^{s}r^{\alpha})\right|^{2}\\
& \cdot\left(\lambda_{k}^{-\frac{1}{4}-\epsilon}\phi_{k}(x),\phi_{k}(x)\right)^{2}dr\\
			\leq& C\lambda^{-2\sigma}_{N+1}\int_{0}^{t_{n}}r^{2-1+2H_{2}-2-2(\sigma+\frac{1-H_{1}}{2}+\epsilon)\alpha/s}dr,
		\end{aligned}
	\end{equation*}
	where we need to require $2H_{2}-2(\sigma+\frac{1-H_{1}}{2}+\epsilon)\alpha/s>0$ and $\sigma+\frac{1-2H_{1}}{2}< s$, i.e., $2\sigma<\min\{\frac{2sH_{2}}{\alpha}+H_{1}-1,2H_{1}+2s-1\}$. Combining the above estimates, the fact that $\lambda_{k}\geq Ck^{2}$ for $k\in\mathbb{N}^{*}$, and the Gr\"{o}nwall inequality, one can obtain the desired results.
\end{proof}

\subsection{Temporal error analyses}

In this subsection, we give the error estimates of the  Mittag--Leffler Euler integrator and the  corresponding fast Mittag--Leffler Euler integrator, respectively.
\begin{theorem}\label{thmtempcon}
	Let $\bar{u}_{N}$ and $\bar{u}^{n}_{N}$ be the solutions of \eqref{eqsemisch0} and \eqref{equfullscheme}, respectively. Then we have
	\begin{equation*}
		\mathbb{E}\|\bar{u}_{N}-\bar{u}^{n}_{N}\|_{L^{2}(D)}^{2}\leq C\tau^{H_{2}+\frac{\alpha(H_{1}-1)}{2s}-\epsilon}.
	\end{equation*}
\end{theorem}
\begin{proof}
	According to \eqref{eqsemisch0sol} and \eqref{equfullscheme}, it yields
	\begin{equation*}
		\begin{aligned}
			&\mathbb{E}\|\bar{u}_{N}-\bar{u}^{n}_{N}\|_{L^{2}(D)}^{2}\\
			\leq&C\mathbb{E}\Bigg\|\sum_{i=1}^{n}\int_{t_{i-1}}^{t_{i}}\mathcal{S}_{R,N}(t_{n}-r)P_{N}(f(\bar{u}_{N}(r))-f(\bar{u}_{N}^{i-1}))dr\Bigg\|_{L^{2}(D)}^{2}\\
			&+C\mathbb{E}\Bigg\|\sum_{i=1}^{n}\int_{t_{i-1}}^{t_{i}}\int_{D}(G_{R,S}(t_{n}-r,x,y)-G_{R,N}(t_{n},r,x,y))\xi^{H_{1},H_{2}}(y,r)dydr\Bigg\|_{L^{2}(D)}^{2}\\
			\leq& \uppercase\expandafter{\romannumeral1}+\uppercase\expandafter{\romannumeral2}.
		\end{aligned}
	\end{equation*}
Combining Assumption \ref{eqnonassump}, Lemma \ref{lemMit2}, and Theorem \ref{thmregnumsol} leads to
	\begin{equation*}
		\begin{aligned}
			\uppercase\expandafter{\romannumeral1}\leq &C\mathbb{E}\Bigg\|\sum_{i=1}^{n}\int_{t_{i-1}}^{t_{i}}\mathcal{S}_{R,N}(t_{n}-r)P_{N}(f(\bar{u}_{N}(r))-f(\bar{u}_{N}(t_{i-1})))dr\Bigg\|_{L^{2}(D)}^{2}\\
			&+C\mathbb{E}\Bigg\|\sum_{i=1}^{n}\int_{t_{i-1}}^{t_{i}}\mathcal{S}_{R,N}(t_{n}-r)P_{N}(f(\bar{u}_{N}(t_{i-1}))-f(\bar{u}_{N}^{i-1}))dr\Bigg\|_{L^{2}(D)}^{2}\\
			\leq& C\tau^{H_{2}+\frac{\alpha(H_{1}-1)}{2s}-\epsilon}+C\sum_{i=1}^{n}\int_{t_{i-1}}^{t_{i}}\mathbb{E}\|f(\bar{u}_{N}(t_{i-1}))-f(\bar{u}_{N}^{i-1})\|_{L^{2}(D)}^{2}dr\\
			\leq& C\tau^{H_{2}+\frac{\alpha(H_{1}-1)}{2s}-\epsilon}+C\tau\sum_{i=1}^{n}\mathbb{E}\|\bar{u}_{N}(t_{i-1})-\bar{u}_{N}^{i-1}\|_{L^{2}(D)}^{2}dr.
		\end{aligned}
	\end{equation*}
	As for $\uppercase\expandafter{\romannumeral2}$, the definitions of $G_{R,S}$ and $G_{R,N}$ and Theorem \ref{thmisometry} imply that
	\begin{equation*}
		\begin{aligned}
			\uppercase\expandafter{\romannumeral2}\leq&C\sum_{k=1}^{N}\left\|{}_{0}\partial^{\frac{1-2H_{2}}{2}}_{r}(E_{\alpha,1}(-\lambda_{k}^{s}(t_{n}-r)^{\alpha})-\mathcal{S}_{R,k}(t_{n},r))\right\|^{2}_{L^{2}((0,t_{n}))}\|
			\phi_{k}(y)\|^{2}_{H^{\frac{1-2H_{1}}{2}}_{0}(D)}.
		\end{aligned}
	\end{equation*}
	Using Lemma \ref{lemEkest}, Remark \ref{Respace}, and the interpolation theorem \cite{Brenner.2008TMToFEM}, one has
	\begin{equation*}
		\begin{aligned}
			\uppercase\expandafter{\romannumeral2}\leq&C\sum_{k=1}^{N}\lambda_{k}^{-\frac{1}{2}-\epsilon}\left\|\lambda_{k}^{\frac{1-H_{1}}{2}}{}_{0}\partial^{\frac{1-2H_{2}}{2}}_{r}\left(E_{\alpha,1}(-\lambda_{k}^{s}(t_{n}-r)^{\alpha})-\mathcal{S}_{R,k}(t_{n},r)\right)\right\|^{2}_{L^{2}((0,t_{n}))}\\
			\leq&C \tau^{H_{2}+\frac{\alpha(H_{1}-1)}{2s}-\epsilon}.
		\end{aligned}
	\end{equation*}
Thus the proof is completed.
\end{proof}

Next, we begin to consider the errors between $u^{n}_{N}$ and $\bar{u}^{n}_{N}$. The following estimate about $\mathcal{E}_{k}(t)$ is needed.

\begin{lemma}\label{lemErrork}
	Let $\mathcal{E}_{k}(t)$ be defined in \eqref{eqdefErrork}. Then it satisfies
	\begin{equation*}
		|\mathcal{E}_{k}(t)|\leq C e^{-\sqrt{2\pi q L}}.
	\end{equation*}
\end{lemma}
\begin{proof}
	According to \cite{Stenger.1981NmboWcosf}, to obtain the desired results, we just need to check that for $r=x+\mathbf{i}y$, $g(r)=(\rho(r))^{\alpha-2}(\rho(r)^{\alpha}+\lambda_{k}^{s})^{-1}\rho^{\prime}(r)$ satisfies the following properties:
	\begin{enumerate}[(a)]
		\item\label{vcona} for some constant $q>0$, the function $g$ is analytic in $D_{q}=\{z\in\mathbb{C}:|\Im (z)|<q\}$, where $\Im(z)$ means the imaginary part of $z$ and
		\begin{equation*}
			\int_{-q}^{q}|g(x+\mathbf{i}y)|dy\rightarrow 0\quad {\rm as}~~x\rightarrow \pm\infty;
		\end{equation*}
		\item\label{vconb} for all $x\in\mathbb{R}$, there holds
		\begin{equation*}
			|g(x)|<Ce^{-|x|};
		\end{equation*}
		\item\label{vconc} \begin{equation*}
			\lim_{y\rightarrow q^{-}}\Bigg(\int_{\mathbb{R}}|g(x+\mathbf{i}y)|dx+\int_{\mathbb{R}}|g(x-\mathbf{i}y)|dx\Bigg)<\infty.
		\end{equation*}
	\end{enumerate}
Now we begin to verify the above properties.
	By the definition of $\rho(r)$, we have $\rho(x+\mathbf{i}y)=\mu(1-\sin(\nu-y)\cosh(x))-\mathbf{i}\mu\cos(\nu-y)\sinh(x)$. It is easy to verify that for $|\nu-y|<\frac{\pi}{2}$, there exists $\rho(x+\mathbf{i}y)\in \Sigma_{\bar{\theta}}$ with $\bar{\theta}\in(\frac{\pi}{2},\pi)$, which leads to that $g(r)$ is analytic in $D_{q}$ with $q\in(0,\frac{\pi}{2}-\nu)$. Moreover, using the facts that $\rho'(x+\mathbf{i}y)\sim \rho(x+\mathbf{i}y)\sim e^{|x|}$ as $|x|\rightarrow \infty$ and $|z^{\alpha-2}(z^{\alpha}+\lambda_{k}^{s})^{-1}|\leq C|z|^{-2}$ \cite{Stenger.1981NmboWcosf} results in \eqref{vcona}, \eqref{vconb}, and \eqref{vconc} directly.
\end{proof}

\begin{theorem}\label{thmfasttempcon}
	Let $\bar{u}^{n}_{N}$ and $u^{n}_{N}$ be the solutions of \eqref{equfullscheme} and \eqref{equfullschemefast}, respectively.  Let $e^{-2\sqrt{2\pi q L}}\sim \mathcal{O}\left(\tau^{2+H_{2}+\frac{\alpha(H_{1}-1)}{2s}-\epsilon}N^{-2\sigma-2+2H_{1}+\epsilon}\right)$ with $2\sigma<\min\{\frac{2sH_{2}}{\alpha}+H_{1}-1,2H_{1}+2s-1\}$, and $f(u)$ satisfies Assumption \ref{eqnonassump}. Then there holds
	\begin{equation*}
		\mathbb{E}\|\bar{u}^{n}_{N}-u^{n}_{N}\|^{2}_{L^{2}(D)}\leq C\tau^{H_{2}+\frac{\alpha(H_{1}-1)}{2s}-\epsilon}N^{-2\sigma}.
	\end{equation*}
\end{theorem}
\begin{proof}
	Introduce
	\begin{equation*}
		\begin{aligned}
			\mathcal{E}^{n-i}_{k}=&\int_{t_{i-1}}^{t_{i}}E_{\alpha,1}(-\lambda_{k}^{s}(t_{n}-r)^{\alpha})dr-\sum_{j=-L}^{L}\omega_{j}(e^{z_{j}t_{n-i+1}}-e^{z_{j}t_{n-i}})z_{j}^{\alpha-2}(z_{j}^{\alpha}+\lambda_{k}^{s})^{-1}.
		\end{aligned}
	\end{equation*}
	According to \eqref{equfullscheme} and \eqref{equfullschemefast}, we have
	\begin{equation*}
		\begin{aligned}
			&\mathbb{E}\|\bar{u}^{n}_{N}-u^{n}_{N}\|^{2}_{L^{2}(D)}\\
			\leq&C\mathbb{E}\Bigg\|\sum_{i=1}^{n}\sum_{k=1}^{N}\Bigg(\int_{t_{i-1}}^{t_{i}}E_{\alpha,1}(-\lambda_{k}^{s}(t_{n}-r)^{\alpha})(P_{N}f(\bar{u}^{i-1}_{N}),\phi_{k}))\phi_{k}dr\\
			&-\sum_{j=-L}^{L}\omega_{j}(e^{z_{j}t_{n-i+1}}-e^{z_{j}t_{n-i}})z_{j}^{\alpha-2}(z_{j}^{\alpha}+\lambda_{k}^{s})^{-1}(P_{N}f(u^{i-1}_{N}),\phi_{k})\phi_{k}\Bigg)\Bigg\|_{L^{2}(D)}^{2}\\
			&+C\mathbb{E}\Bigg\|\sum_{i=1}^{n}\sum_{k=1}^{N}\mathcal{E}^{n-i}_{k}\left(\xi^{H_{1},H_{2}}_{R}(x,t_{i}),\phi_{k}\right)\phi_{k}\Bigg\|_{L^{2}(D)}^{2}\\
			\leq&C\mathbb{E}\Bigg\|\sum_{i=1}^{n}\int_{t_{i-1}}^{t_{i}}\mathcal{S}_{RN}(t_{n}-r)P_{N}(f(\bar{u}^{i-1}_{N})-f(u^{i-1}_{N}))dr\Bigg\|_{L^{2}(D)}^{2}\\
			&+C\mathbb{E}\Bigg\|\sum_{i=1}^{n}\sum_{k=1}^{N}\mathcal{E}^{n-i}_{k}(P_{N}f(u^{i-1}_{N}),\phi_{k})\phi_{k}\Bigg\|_{L^{2}(D)}^{2}\\
			&+C\mathbb{E}\Bigg\|\sum_{i=1}^{n}\sum_{k=1}^{N}\mathcal{E}^{n-i}_{k}\left(\xi^{H_{1},H_{2}}_{R}(x,t_{i}),\phi_{k}\right)\phi_{k}\Bigg\|_{L^{2}(D)}^{2}\\
			\leq&\uppercase\expandafter{\romannumeral1}+\uppercase\expandafter{\romannumeral2}+\uppercase\expandafter{\romannumeral3}.
		\end{aligned}
	\end{equation*}
	Using Lemma \ref{lemMit2}, one can get
	\begin{equation*}
		\begin{aligned}
			\uppercase\expandafter{\romannumeral1}\leq& C\sum_{i=1}^{n}\int_{t_{i-1}}^{t_{i}}\mathbb{E}\|f(\bar{u}^{i-1}_{N})-f(u^{i-1}_{N})\|_{L^{2}(D)}^{2}dr\\
			\leq& C\sum_{i=1}^{n}\int_{t_{i-1}}^{t_{i}}\mathbb{E}\|\bar{u}^{i-1}_{N}-u^{i-1}_{N}\|_{L^{2}(D)}^{2}dr
		\end{aligned}
	\end{equation*}
	and
	\begin{equation*}
		\begin{aligned}
			\uppercase\expandafter{\romannumeral2}\leq&Ce^{-2\sqrt{2\pi q L}}\mathbb{E}\left\|\sum_{i=1}^{n}P_{N}f(u_{N}^{i-1})\right\|_{L^{2}(D)}^{2}\\
			\leq&Ce^{-2\sqrt{2\pi q L}}\mathbb{E}\left\|\sum_{i=1}^{n}\left(f(u_{N}^{i-1})-f(\bar{u}_{N}^{i-1})\right)\right\|_{L^{2}(D)}^{2}+Ce^{-2\sqrt{2\pi d L}}\mathbb{E}\left\|\sum_{i=1}^{n}P_{N}f(\bar{u}_{N}^{i-1})\right\|_{L^{2}(D)}^{2}\\
			\leq&Ce^{-2\sqrt{2\pi q L}}\mathbb{E}\left(\sum_{i=1}^{n}\left\|u_{N}^{i-1}-\bar{u}_{N}^{i-1}\right\|_{L^{2}(D)}\right)^{2}+Ce^{-2\sqrt{2\pi d L}}\mathbb{E}\left\|\sum_{i=1}^{n}P_{N}f(\bar{u}_{N}^{i-1})\right\|_{L^{2}(D)}^{2}\\
			\leq& Ce^{-2\sqrt{2\pi q L}}\tau^{-2}\sum_{i=1}^{n}\int_{t_{i-1}}^{t_{i}}\mathbb{E}\left\|\bar{u}^{i-1}_{N}-u^{i-1}_{N}\right\|_{L^{2}(D)}^{2}dr+Ce^{-2\sqrt{2\pi d L}}\tau^{-2}.
		\end{aligned}
	\end{equation*}
	According to Lemma \ref{lemErrork}, there holds
	\begin{equation*}
		\begin{aligned}
			\uppercase\expandafter{\romannumeral3}\leq& Ce^{-2\sqrt{2\pi d L}} \mathbb{E}\Bigg\|\sum_{i=1}^{n}\sum_{k=1}^{N}\left(\xi^{H_{1},H_{2}}_{R}(x,t_{i}),\phi_{k}\right)\phi_{k}\Bigg\|_{L^{2}(D)}^{2}\\
			\leq&Ce^{-2\sqrt{2\pi q L}} \mathbb{E}\Bigg\|\sum_{i=1}^{n}\sum_{k=1}^{N}\frac{1}{\tau}\int_{I_{i}}\int_{D}\phi_{k}(y)\xi^{H_{1},H_{2}}(y,r)dydr\phi_{k}(x)\chi_{I_{i}}(t)\Bigg\|_{L^{2}(D)}^{2}\\
			\leq&Ce^{-2\sqrt{2\pi q L}}\tau^{-2}\sum_{k=1}^{N}\left\|{}_{0}\partial_{t}^{\frac{1-2H_{2}}{2}}1\right\|_{L^{2}(0,T)}^{2}\|\phi_{k}(y)\|_{H^{\frac{1-2H_{1}}{2}}_{0}(D)}^{2}\\
			\leq&Ce^{-2\sqrt{2\pi q L}}\tau^{-2}\sum_{k=1}^{N}\lambda_{k}^{-\frac{1}{2}-\epsilon}\left\|{}_{0}\partial_{t}^{\frac{1-2H_{2}}{2}}1\right\|_{L^{2}(0,T)}^{2}\lambda_{k}^{1-H_{1}+\epsilon}\\
			\leq&Ce^{-2\sqrt{2\pi q L}}\tau^{-2}\lambda_{N}^{1-H_{1}+\epsilon}.
		\end{aligned}
	\end{equation*}
After gathering the above estimates, the desired results can be achieved.
\end{proof}


\section{Numerical experiments}
In this section, some numerical experiments are performed to validate the convergence and efficiency of our numerical scheme. Here we choose $f(u)=\sin(u)$.  Since the exact solution is unknown, we measure the convergence rates by
\begin{equation*}
	{\rm Rate}=\frac{\ln(e_{N}/e_{2N})}{\ln(2)},\quad {\rm Rate}=\frac{\ln(e_{\tau}/e_{\tau/2})}{\ln(2)},
\end{equation*}
where
\begin{equation*}
	\begin{aligned}
		&e_{N}=\left (\frac{1}{l}\sum_{i=1}^{l}\left\|u^{M}_{N}(\omega_{i})-u^{M}_{2N}(\omega_{i})\right\|^{2}_{L^{2}(D)}\right )^{1/2},\\
		&e_{\tau}=\left (\frac{1}{l}\sum_{i=1}^{l}\|u_{\tau}(\omega_{i})-u_{\tau/2}(\omega_{i})\|^{2}_{L^{2}(D)}\right )^{1/2}
	\end{aligned}
\end{equation*}
with $u^{M}_{N}$ and $u_{\tau}$ being the solutions at time $T=t_{M}$ with mesh size $h$ and time step size $\tau$, respectively.

In the numerical experiments, we take $\mu=7$, $\nu=0.1\pi$, and $q=0.05\pi$.
\begin{example}
	In this example, we take $l=100$, $L=200$, and $T=0.1$. To show the temporal convergence, we take $N=256$ to eliminate the influence from spatial discretization. We show the errors and convergence rates with different $\alpha$, $s$, $H_{1}$, and $H_{2}$ in Table \ref{tab:time} (the numbers in the bracket in the last column denote the theoretical rates predicted by Theorem \ref{thmspacon}). It is easy to see that all the numerical results are the same with our predicted ones.
	\begin{table}[htbp]
		\caption{Temporal errors and convergence rates}
		\begin{tabular}{ccccccl}
			\hline
			($H_1$,$H_2$,$s$,$\alpha$) & 8 & 16 & 32 & 64&128 & Rate \\
			\hline
			(0.2,0.2,0.7,0.3) & 6.603E-02 & 6.582E-02 & 6.975E-02 & 6.407E-02&6.002E-02 & $\approx0.0344(=0.0286)$ \\
			(0.2,0.5,0.7,0.7) & 8.488E-02 & 8.230E-02 & 8.014E-02 & 6.801E-02&6.300E-02 & $\approx0.1075(=0.1)$ \\
			(0.3,0.4,0.7,0.3) & 1.921E-02 & 1.697E-02 & 1.329E-02 & 1.139E-02&1.015E-02 & $\approx0.2299(=0.25)$ \\
			(0.4,0.5,0.4,0.3) & 3.580E-02 & 3.310E-02 & 2.342E-02 & 1.956E-02&1.624E-02 & $\approx0.2851(=0.275)$ \\
			(0.5,0.4,0.4,0.3) & 4.563E-02 & 4.250E-02 & 3.380E-02 & 3.195E-02&2.414E-02 & $\approx0.2296(=0.2125)$ \\
			(0.5,0.5,0.7,0.6) & 3.224E-02 & 2.737E-02 & 2.079E-02 & 1.786E-02&1.402E-02 & $\approx0.3004(=0.2857)$ \\
			\hline
		\end{tabular}
		\label{tab:time}
	\end{table}
\end{example}

\begin{example}
	In this example, we take $l=100$, $L=200$, and $T=0.1$. To show the spatial convergence, we take $\tau=T/2048$ to avoid the influence from temporal discretization. The numerical results with different $\alpha$, $s$, $H_{1}$, and $H_{2}$ are presented in Table \ref{tab:spac}, where the numbers in the bracket in the last column denote the theoretical rates predicted by Theorem \ref{thmspacon}. All the results agree with the predicted ones and all the convergence rates are optimal.
	\begin{table}[htbp]
		\caption{Spatial errors and convergence rates}
		\begin{tabular}{ccccccl}
			\hline
			($H_1$,$H_2$,$s$,$\alpha$) & 4 & 8 & 16 & 32&64 & Rate \\
			\hline
			(0.2,0.5,0.6,0.3) & 6.717E-02 & 5.239E-02 & 4.238E-02 & 3.294E-02&2.461E-02 & $\approx0.3621(=0.4)$ \\
			(0.2,0.5,0.7,0.6) & 1.069E-01 & 8.927E-02 & 6.927E-02 & 5.303E-02&3.785E-02 & $\approx0.3746(=0.3667)$ \\
			(0.4,0.5,0.4,0.6) & 2.812E-01 & 2.985E-01 & 2.806E-01 & 2.797E-01&2.751E-01 & $\approx0.008(=0.0667)$ \\
			(0.5,0.3,0.4,0.4) & 2.725E-01 & 2.569E-01 & 2.470E-01 & 2.449E-01&2.105E-01 & $\approx0.0932(=0.1)$ \\
			(0.5,0.4,0.8,0.2) & 1.092E-02 & 5.746E-03 & 3.253E-03 & 1.518E-03&6.773E-04 & $\approx1.0027(=1.1)$ \\
			(0.5,0.4,0.9,0.2) & 8.053E-03 & 3.599E-03 & 1.434E-03 & 5.976E-04&2.444E-04 & $\approx1.2606(=1.3)$ \\
			\hline
		\end{tabular}
		\label{tab:spac}
	\end{table}
\end{example}

\begin{example}
	Furthermore, to show the efficiency of our algorithm, we take $T=0.1$, $L=200$, $\alpha=0.7$, $s=0.5$, and $H_{1}=H_{2}=0.5$. The CPU times with different $M$ are shown in Fig. \ref{fig:cputime}. It can be noted that our provided fast algorithm (see \eqref{equfullschemefast}) is more effective  as $M$ becomes larger, and the growth of CPU time of the classical algorithm (see \eqref{equfullscheme}) is nearly twice as much as the one of fast algorithm.
	\begin{figure}[tbph]
		\centering
		\includegraphics[width=0.7\linewidth]{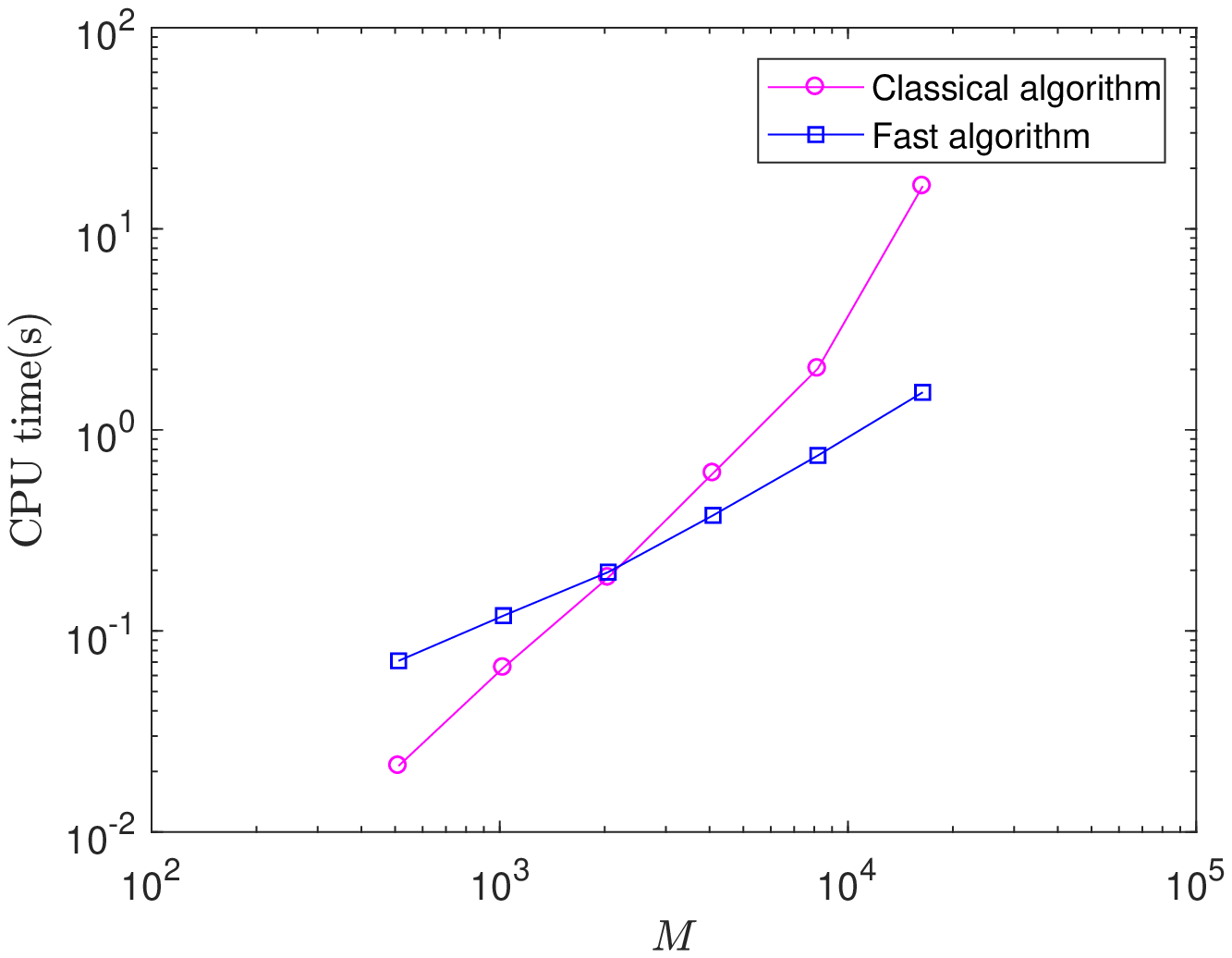}
		\caption{}
		\label{fig:cputime}
	\end{figure}
\end{example}

\section{Conclusions}

Taking fractional Brownian sheet as the external noise of the model governing the probability density function of the competition dynamics between super- and sub- diffusions, we discuss the numerical scheme of the built  stochastic differential equation.  Based on the regularity of the solution of the equation, we first provide a new Wong--Zakai approximation for fractional Brownian sheet and obtain the optimal convergence.  An efficient fully discrete scheme is developed by using spectral Galerkin method in space and fast Mittag--Leffler Euler integrator in time.  The strict error analyses of the fully discrete scheme are presented, and the proposed numerical examples validate the effectiveness of the algorithm.

%
%
%

\bibliographystyle{siamplain}
\bibliography{ref}

\end{document}